\documentclass[a4paper, intlimits, reqno]{amsart}

\usepackage[english]{babel}
\usepackage[T1]{fontenc}
\usepackage[utf8]{inputenc}

\usepackage{amsmath}
\usepackage{amssymb}
\usepackage{MnSymbol}
\usepackage{amsthm}
\allowdisplaybreaks
\usepackage{amsfonts}
\usepackage{mathrsfs} 
\usepackage{mathtools}
\usepackage[all]{xy}
\usepackage{nicefrac}
\usepackage{enumitem}

\usepackage{multicol}
\usepackage{url}
\usepackage{dsfont}
\usepackage[numbers,sort&compress]{natbib}
\usepackage{doi}
\usepackage{prettyref}
\usepackage{xcolor}
\usepackage{orcidlink}

\newrefformat{defn}{Definition \ref{#1}}
\newrefformat{rem}{Remark \ref{#1}}
\newrefformat{sect}{Section \ref{#1}}
\newrefformat{sub}{Section \ref{#1}}
\newrefformat{prop}{Proposition \ref{#1}}
\newrefformat{thm}{Theorem \ref{#1}}
\newrefformat{cor}{Corollary \ref{#1}}
\newrefformat{ex}{Example \ref{#1}}

\swapnumbers
\newtheoremstyle{dotless}{}{}{\itshape}{}{\bfseries}{}{}{}
\theoremstyle{dotless}
\theoremstyle{plain}
\newtheorem{thm}{Theorem}[section]

\newtheorem{prop}[thm]{Proposition}
\newtheorem{cor}[thm]{Corollary}
\theoremstyle{definition}
\newtheorem{defn}[thm]{Definition}
\newtheorem{rem}[thm]{Remark}
\newtheorem{exa}[thm]{Example}
\newcommand{\N} {\mathbb{N}}

\newcommand{\R} {\mathbb{R}}
\newcommand{\C} {\mathbb{C}}
\newcommand{\K} {\mathbb{K}}
\newcommand{\D} {\mathbb{D}}
\newcommand{\F} {\mathcal{F}(\Omega)}
\newcommand{\FE} {\mathcal{F}(\Omega,E)}
\DeclareMathOperator{\id}{id}

\makeatletter
\newcommand{\fakephantomsection}{%
  \Hy@GlobalStepCount\Hy@linkcounter%
  \Hy@MakeCurrentHref{\@currenvir.\the\Hy@linkcounter}
  \Hy@raisedlink{\hyper@anchorstart{\@currentHref}\hyper@anchorend}%
}
\makeatother

\begin{document}

\title[On linearisation and uniqueness of preduals]{On linearisation and uniqueness of preduals}
\author[K.~Kruse]{Karsten Kruse\,\orcidlink{0000-0003-1864-4915}}
\address[Karsten Kruse]{University of Twente, Department of Applied Mathematics, P.O. Box 217, 7500 AE Enschede, The Netherlands, and Hamburg University of Technology, Institute of Mathematics, Am Schwarzenberg-Campus~3, 21073 Hamburg, Germany}

\email{k.kruse@utwente.nl}



\date{\today}
\begin{abstract}
We study strong linearisations and the uniqueness of preduals of locally convex Hausdorff spaces of scalar-valued functions. 
Strong linearisations are special preduals. A locally convex Hausdorff space $\F$ of scalar-valued functions on a non-empty 
set $\Omega$ is said to admit a \emph{strong linearisation} if there are a locally convex Hausdorff space $Y$, 
a map $\delta\colon\Omega\to Y$ and a topological isomorphism $T\colon\F\to Y_{b}'$ such that $T(f)\circ \delta= f$ for all $f\in\F$. We give sufficient conditions that allow us to lift strong linearisations from the scalar-valued 
to the vector-valued case, covering many previous results on linearisations, and 
use them to characterise the bornological spaces $\F$ with (strongly) unique predual in certain classes of locally convex Hausdorff spaces.
\end{abstract}
\maketitle

\section{Introduction}\label{sect:intro}

The purpose of this paper is twofold. First, we present a general mechanism how to transfer a strong linearisation of a locally convex Hausdorff space $\F$ of scalar-valued functions on a non-empty set $\Omega$ to a vector-valued counterpart. 
Second, we characterise those spaces among the spaces $\F$ which have a (strongly) unique predual in certain classes of locally convex Hausdorff spaces.  
Strong linearisations are special preduals. Recall that a locally convex Hausdorff space $X$ is called a \emph{dual space} and 
a tuple $(Y,\varphi)$ a \emph{predual} of $X$ if the tuple consists of a locally convex Hausdorff space $Y$ 
and a topological isomorphism $\varphi\colon X\to Y_{b}'$ where $Y_{b}'$ is the strong dual of $Y$. 
In \cite{kruse2023a} we derived necessary and sufficient conditions for the existence of preduals and strong linearisations 
of bornological spaces such that the predual has certain properties like being complete and barrelled, a DF-space, a Fr\'echet space or completely normable (cf.~\prettyref{prop:predual_complete} and \prettyref{cor:scb_linearisation}). A \emph{strong linearisation} of a locally convex Hausdorff space $\F$ of $\K$-valued functions on a non-empty set $\Omega$ where $\K=\R$ or $\C$ is a triple $(\delta,Y,T)$ of a locally convex Hausdorff space $Y$ over the field 
$\K$, a map $\delta\colon\Omega\to Y$ and a topological isomorphism $T\colon\F\to Y_{b}'$ if $T(f)\circ \delta= f$ for all $f\in\F$ (see \cite[p.~683]{carando2004}, \cite[p.~181, 184]{jaramillo2009} and 
\cite[Proposition 2.6, p.~1595]{kruse2023a}). 

We show that one can lift a strong linearisation $(\delta,Y,T)$ of the scalar-valued case $\F$ to the vector-valued case 
in \prettyref{thm:linearisation_full} and unify preceding results on strong vector-valued linearisations 
from Aron, Dimant, Garc\'{i}a-Lirola and Maestre \cite{aron2024}, Bonet, Doma\'nski and Lindstr\"om \cite{bonet2001}, 
Gupta and Baweja \cite{gupta2016}, Jord\'a \cite{jorda2013}, Laitila and Tylli \cite{laitila2006}, Mujica \cite{mujica1991} 
and Quang \cite{quang2023} where $\F$ is a weighted Banach space of holomorphic 
or harmonic functions, and from Beltr\'an \cite{beltran2012,beltran2014}, Bierstedt, Bonet and Galbis \cite{bierstedt1993}, 
Bonet and Friz \cite{bonet2002} and Galindo, Garc\'ia and Maestre \cite{galindo1992} where $\F$ is a weighted bornological space of 
holomorphic functions, and Grothendieck \cite{grothendieck1966} where $\F$ is the space of continuous bilinear forms on the 
product $\Omega\coloneqq F\times G$ of two locally convex Hausdorff spaces $F$ and $G$. 
Further, our results on strong linearisations augment results on continuous linearisations, where 
$\delta$ is continuous but $T$ need not be continuous, by Carando and Zalduendo \cite{carando2004} and 
Jaramillo, Prieto and Zalduendo \cite{jaramillo2009}. Linearisations are a useful tool since they identify (usually) non-linear functions $f$ with (continuous) linear operators $T(f)$ and thus allow to apply linear functional analysis to non-linear functions. They are often used to transfer results that are known for scalar-valued functions to vector-valued functions, 
see e.g.~\cite{bonet2001,bonet2002,grothendieck1966,jaramillo2009,jorda2013,laitila2006}. 
We give an example of such an application to an extension problem in the case that $\F$ 
is a complete bornological DF-space in \prettyref{thm:extension}. 

However, our main motivation in considering strong linearisations in the present article does not stem from transferring results 
from the scalar-valued to the vector-valued case. We use strong linearisations to study the question whether a predual is \emph{unique} 
up to identification via topological isomorphisms. This question is usually only treated in the case of Banach spaces and even then 
mostly in the isometric setting, i.e.~in the case of Banach spaces $X$ having a predual $(Y,\varphi)$ such that $Y$ is a 
Banach space and $\varphi$ an isometric isomorphism. We refer the reader to the thorough survey of Godefroy \cite{godefroy1989} 
in the isometric setting and to the paper \cite{brown1975} of Brown and Ito in the non-isometric Banach setting. 
Since we are interested in the general setting of locally convex Hausdorff spaces, our isomorphisms $\varphi$ are in general only topological and we also have to choose a class $\mathcal{C}$ of locally convex Hausdorff spaces in which we strive for 
uniqueness of the predual. Such a class $\mathcal{C}$ has to be closed under topological isomorphisms because 
otherwise there is no hope for uniqueness of the predual up to identification via topological isomorphisms. 
Using such a class $\mathcal{C}$, it is possible to introduce two notions of uniqueness of a predual that are already known 
in the Banach setting. First, we say that a dual space $X$ has a \emph{unique $\mathcal{C}$ predual} if for all preduals 
$(Y,\varphi)$ and $(Z,\psi)$ of $X$ such that $Y,Z\in\mathcal{C}$ 
there is a topological isomorphism $\lambda\colon Y\to Z$. Second, we say that a dual space $X$ has a 
\emph{strongly unique $\mathcal{C}$ predual} if for all preduals $(Y,\varphi)$ and $(Z,\psi)$ such that $Y,Z\in\mathcal{C}$ 
and all topological isomorphisms $\alpha\colon Z_{b}'\to Y_{b}'$ there is a topological isomorphism $\lambda\colon Y\to Z$ 
such that $\lambda^{t}=\alpha$. The second definition might seem a bit strange at first but we will see that 
it takes the topological isomorphisms of a predual into account, so it does not forget the additional structure, 
and it fits quite well to strong linearisations. 
Strong uniqueness of a predual in $\mathcal{C}$ means that all preduals of $X$ are \emph{equivalent} 
in the sense that for all preduals $(Y,\varphi)$ and $(Z,\psi)$ of $X$ such that $Y,Z\in\mathcal{C}$ 
there is a topological isomorphism $\lambda\colon Y\to Z$ such that $\lambda^{t}=\varphi\circ\psi^{-1}$, 
see \prettyref{prop:strongly_unique_equivalent}. 
We note that this definition of the equivalence of preduals is an adaptation of a corresponding definition in the Banach setting 
by Gardella and Thiel \cite{gardella2020}. The notion of equivalence of preduals allows us on the one hand to show 
in \prettyref{cor:carando_zalduendo_predual} how our construction of a strong linearisation is 
related to the continuous linearisation of Carando and Zalduendo, and on the other hand to characterise the bornological spaces 
$\F$ of $\K$-valued functions which have a (strongly) unique $\mathcal{C}$ predual, 
see \prettyref{cor:strongly_unique_C_predual_without_lin_ind} and \prettyref{cor:unique_C_predual_without_lin_ind}. 
We refer the reader who is also interested in the corresponding results of the present paper in the isometric Banach setting 
to \cite{kruse2023c}.

\section{Notions and preliminaries}
\label{sect:notions}

In this short section we recall some basic notions from the theory of locally convex spaces and present some prelimary results 
on dual spaces and their preduals (cf.~\cite[Section 2]{kruse2023a}). 
For a locally convex Hausdorff space $X$ over the field $\K\coloneqq\R$ or $\C$ we denote by $X'$ the topological linear dual space 
and by $U^{\circ}$ the \emph{polar set} of a subset $U\subset X$. 
If we want to emphasize the dependency on the locally convex Hausdorff topology $\tau$ of $X$, we write $(X,\tau)$ and 
$(X,\tau)'$ instead of just $X$ and $X'$, respectively. We denote by $\sigma(X',X)$ the topology on $X'$ of uniform convergence 
on finite subsets of $X$ and by $\beta(X',X)$ the topology on $X'$ of uniform convergence on bounded subsets of $X$. 
Further, we set $X_{b}'\coloneqq (X',\beta(X',X))$. 
For a continuous linear map $T\colon X\to Y$ between two locally convex Hausdorff spaces $X$ and $Y$ we denote by 
$T^{t}\colon Y'\to X'$, $y'\mapsto y'\circ T$, the \emph{dual map} of $T$ and write $T^{tt}\coloneqq (T^{t})^{t}$. 
Furthermore, we say that a linear map $T\colon X\to Y$ between two locally convex Hausdorff spaces $X$ and $Y$ is \emph{(locally) 
bounded} if it maps bounded sets to bounded sets. 
Moreover, for two locally convex Hausdorff topologies $\tau_{0}$ and $\tau_{1}$ 
on $X$ we write $\tau_{0}\leq\tau_{1}$ if $\tau_{0}$ is coarser than $\tau_{1}$. For a normed space $(X,\|\cdot\|)$ we denote by 
$B_{\|\cdot\|}\coloneqq\{x\in X\;|\;\|x\|\leq 1\}$ the $\|\cdot\|$-closed unit ball of $X$. Further, we write 
$\tau_{\operatorname{co}}$ for the \emph{compact-open topology}, i.e.~the topology of uniform convergence on compact subsets of 
$\Omega$, on the space $\mathcal{C}(\Omega)$ of $\K$-valued continuous functions on a topological Hausdorff space $\Omega$. 
In addition, we write $\tau_{\operatorname{p}}$ for the \emph{topology of pointwise convergence} on the space $\K^{\Omega}$ of $\K$-valued functions on a set $\Omega$. By a slight abuse of notation we also use the symbols $\tau_{\operatorname{co}}$ and 
$\tau_{\operatorname{p}}$ for the relative compact-open topology and the relative topology of pointwise convergence on topological subspaces of $\mathcal{C}(\Omega)$ and $\K^{\Omega}$, respectively. 
For further unexplained notions on the theory of locally convex Hausdorff spaces we refer the reader to \cite{jarchow1981,meisevogt1997,bonet1987}.

\begin{defn}\label{defn:predual_equivalent}
Let $X$ be a locally convex Hausdorff space. 
\begin{enumerate}
\item[(a)] We call $X$ a \emph{dual space} if there are a locally convex Hausdorff space $Y$ and a topological isomorphism $\varphi\colon X\to Y_{b}'$. The tuple $(Y,\varphi)$ is called a \emph{predual} of $X$. 
\item[(b)] Let $X$ be a dual space. We say that two preduals $(Y,\varphi)$ and $(Z,\psi)$ of $X$ are \emph{equivalent} and write $(Z,\psi)\sim (Y,\varphi)$ if there is a topological isomorphism 
$\lambda\colon Y\to Z$ such that $\lambda^{t}=\varphi\circ\psi^{-1}$. 
\end{enumerate} 
\end{defn}

The preceding definition of a predual is already given in \cite[Definition 2.1, p.~1593]{kruse2023a} 
and for Banach spaces e.g.~in \cite[p.~321]{brown1975}. In the setting of Banach spaces the definition 
of the equivalence of preduals is given in \cite[Definition 2.1]{gardella2020}. It is also easily checked that 
$\sim$ actually defines an equivalence relation on the family of all preduals of a dual space. 
Further, if we have a predual $(Y,\varphi)$ of a dual space $X$ and another locally convex Hausdorff space $Z$ 
which is topologically isomorphic to $Y$, we can always augment $Z$ to a predual which is equivalent to 
$(Y,\varphi)$. 

\begin{rem}\label{rem:augment_isom_to_equiv_predual}
Let $X$ be a dual space with predual $(Y,\varphi)$ and $Z$ a locally convex Hausdorff space. 
If there is a topological isomorphism $\lambda\colon Y\to Z$, then 
the map $\varphi_{\lambda}\colon X\to Z_{b}'$, $\varphi_{\lambda}\coloneqq 
(\lambda^{-1})^{t}\circ\varphi$, is a topological isomorphism with 
$\varphi_{\lambda}^{-1}(z')=\varphi^{-1}(z'\circ\lambda)$ for all $z'\in Z'$ and $\lambda^{t}=\varphi\circ\varphi_{\lambda}^{-1}$. In particular, 
$(Z,\varphi_{\lambda})$ is a predual of $X$ and $(Z,\varphi_{\lambda})\sim (Y,\varphi)$.
\end{rem}

If $X$ is a dual space with a quasi-barrelled predual, we may consider this predual as a 
topological subspace of the strong dual of $X$.

\begin{prop}[{\cite[Proposition 2.2, p.~1593]{kruse2023a}}]\label{prop:predual_into_dual}
Let $X$ be a dual space with quasi-barrelled predual $(Y,\varphi)$.
Then the map 
\[
\Phi_{\varphi}\colon Y\to X_{b}',\;y\longmapsto[x \mapsto \varphi(x)(y)],
\]
is a topological isomorphism into, i.e.~a topological isomorphism to its range. 
\end{prop}

Now, if we want to study whether a dual space has a \emph{unique} predual by identification via 
topological isomorphisms, we have to restrict the range of preduals we consider 
because even a Banach space may have a predual which
is also a Banach space, and another predual which is not a Banach space 
(see \cite[Example 3.15, p.~1601]{kruse2023a}). Two such preduals cannot be topologically isomorphic. 

\begin{defn}\label{defn:unique_predual}
Let $X$ be a dual space and $\mathcal{C}$ a class of locally convex Hausdorff spaces that 
is \emph{closed under topological isomorphisms}, i.e.~if $Y\in\mathcal{C}$ and $Z$ is a locally convex Hausdorff space which is 
topologically isomorphic to $Y$, then $Z\in\mathcal{C}$. 
\begin{enumerate}
\item[(a)] We say that $X$ has a \emph{unique $\mathcal{C}$ predual} if for all preduals $(Y,\varphi)$ and $(Z,\psi)$ of $X$ such that $Y,Z\in\mathcal{C}$ 
there is a topological isomorphism $\lambda\colon Y\to Z$. 
\item[(b)] We say that $X$ has a \emph{strongly unique $\mathcal{C}$ predual} if for all preduals 
$(Y,\varphi)$ and $(Z,\psi)$ such that $Y,Z\in\mathcal{C}$ 
and all topological isomorphisms $\alpha\colon Z_{b}'\to Y_{b}'$ 
there is a topological isomorphism $\lambda\colon Y\to Z$ such that $\lambda^{t}=\alpha$. 
\end{enumerate} 
\end{defn}

In the context of dual Banach spaces where $\mathcal{C}$ is the class of Banach spaces \prettyref{defn:unique_predual} (a) is already given in e.g.~\cite[p.~321]{brown1975} (to be more precise, in the setting of dual Banach spaces, $\mathcal{C}$ in \prettyref{defn:unique_predual} is not the class of Banach spaces but of completely normable spaces since $Z$ in the definition of closedness under 
topological isomorphisms is a locally convex Hausdorff space which need not be a normed space initially). 
\prettyref{defn:unique_predual} (b) is inspired by a similar definition of 
a strongly unique \emph{isometric} Banach predual of a Banach space in e.g.~\cite[p.~134]{godefroy1989} 
and \cite[p.~469]{weaver2018a}. We frequently consider four classes $\mathcal{C}$ of preduals 
in the present paper: the class of complete barrelled locally convex Hausdorff spaces, the class of complete barrelled DF-spaces, the class of Fr\'echet spaces and the class of completely normable spaces. 

\begin{rem}\label{rem:strongly_unique}
Let $X$ be a dual space, $\mathcal{C}_{1}$ and $\mathcal{C}_{2}$ be classes 
of locally convex Hausdorff spaces that are closed under topological isomorphisms and 
$\mathcal{C}_{1}$ be contained in $\mathcal{C}_{2}$. If $X$ has a (strongly) 
unique $\mathcal{C}_{2}$ predual and there is a predual $(Y,\varphi)$ of $X$ such 
that $Y\in\mathcal{C}_{1}$, then $X$ has also a (strongly) unique $\mathcal{C}_{1}$ predual.
\end{rem}

\begin{prop}\label{prop:unique_equivalent}
Let $X$ be a dual space and $\mathcal{C}$ a class of locally convex Hausdorff spaces that 
is closed under topological isomorphisms. Then the following assertions are equivalent.
\begin{enumerate}
\item[(a)] $X$ has a unique $\mathcal{C}$ predual. 
\item[(b)] For all preduals $(Y,\varphi)$ and $(Z,\psi)$ of $X$ such that $Y,Z\in\mathcal{C}$, there 
is a topological isomorphism $\mu\colon X\to Z_{b}'$ such that $(Z,\mu)\sim(Y,\varphi)$. 
\end{enumerate}
\end{prop}
\begin{proof}
Let $(Y,\varphi)$ and $(Z,\psi)$ be preduals of $X$ such that $Y,Z\in\mathcal{C}$.

(a)$\Rightarrow$(b) Since $X$ has a unique predual, there is a topological isomorphism $\lambda\colon Y\to Z$. Now, statement (b) follows from \prettyref{rem:augment_isom_to_equiv_predual} with 
$\mu\coloneqq\varphi_{\lambda}$.

(b)$\Rightarrow$(a) Let $\mu\colon X\to Z_{b}'$ be a topological isomorphism such that 
$(Z,\mu)\sim(Y,\varphi)$. Then there is a topological isomorphism $\lambda\colon Y\to Z$ such that 
$\lambda^{t}=\varphi\circ\mu^{-1}$. Thus $X$ has a unique $\mathcal{C}$ predual.
\end{proof}

\begin{prop}\label{prop:strongly_unique_equivalent}
Let $X$ be a dual space and $\mathcal{C}$ a class of locally convex Hausdorff spaces that 
is closed under topological isomorphisms. Then the following assertions are equivalent.
\begin{enumerate}
\item[(a)] $X$ has a strongly unique $\mathcal{C}$ predual. 
\item[(b)] All preduals of $X$ in $\mathcal{C}$ are equivalent 
(in the sense of \prettyref{defn:predual_equivalent} (b)).
\end{enumerate}
\end{prop}
\begin{proof}
Let $(Y,\varphi)$ and $(Z,\psi)$ be preduals of $X$ such that $Y,Z\in\mathcal{C}$.

(a)$\Rightarrow$(b) The map $\alpha\colon Z_{b}'\to Y_{b}'$, $\alpha\coloneqq \varphi\circ\psi^{-1}$, 
is a topological isomorphism. Since $X$ has a strongly unique predual, there is a topological isomorphism $\lambda\colon Y\to Z$ 
such that $\lambda^{t}=\alpha=\varphi\circ\psi^{-1}$. Thus we have $(Z,\psi)\sim (Y,\varphi)$.

(b)$\Rightarrow$(a) Let $\alpha\colon Z_{b}'\to Y_{b}'$ be a topological isomorphism. 
The tuple $(Y,\alpha\circ\psi)$ is also a predual of $X$. Since all preduals are equivalent, there is a topological isomorphism 
$\lambda\colon Y\to Z$ such that $\lambda^{t}=(\alpha\circ\psi)\circ\psi^{-1}=\alpha$. Thus $X$ has a strongly unique 
$\mathcal{C}$ predual.
\end{proof}

In the context of dual Banach spaces \prettyref{prop:strongly_unique_equivalent} (b) is used 
in \cite[Definition 2.7]{gardella2020} to give an equivalent definition of a dual Banach space 
having a strongly unique Banach predual. Clearly, reflexive locally convex Hausdorff spaces are dual 
spaces. The bornological ones among them also have a strongly unique complete barrelled predual.

\begin{exa}
Let $X$ be a reflexive bornological locally convex Hausdorff space. Then $X$ has a strongly 
unique complete barrelled predual. In particular, every complete barrelled predual of $X$ is reflexive.
\end{exa}
\begin{proof}
Since $X$ is reflexive and bornological, the strong dual $X_{b}'$ is a reflexive, hence barrelled, 
predual of $X$ (equipped with the canonical evaluation map 
$\mathcal{J}_{X}\colon X\to (X_{b}')_{b}'$, $x\longmapsto [x'\mapsto x'(x)]$) 
and also complete by \cite[\S 39, 6.(4), p.~143]{koethe1979}. 
Let $(Y,\varphi)$ be another complete barrelled predual of the reflexive space $X$. 
Then $Y_{b}'$ is reflexive and thus $(Y_{b}')_{b}'$ as well. Since $Y$ is complete and barrelled, it is a closed subspace 
of $(Y_{b}')_{b}'$ via the map $\mathcal{J}_{Y}$ by \cite[11.2.2 Proposition, p.~222]{jarchow1981}, implying that $Y$ is reflexive 
by \cite[11.5.5 Proposition (a), p.~228]{jarchow1981}. Therefore the map $\lambda\colon X_{b}'\to Y$, 
$\lambda\coloneqq \mathcal{J}_{Y}^{-1}\circ(\mathcal{J}_{X}\circ\varphi^{-1})^{t}\circ\mathcal{J}_{X_{b}'}$, is a topological 
isomorphism and 
\begin{align*}
  \lambda^{t}
&=\mathcal{J}_{X_{b}'}^{t}\circ(\mathcal{J}_{X}\circ\varphi^{-1})^{tt}\circ(\mathcal{J}_{Y}^{-1})^{t}
 =\mathcal{J}_{(X_{b}')_{b}'}^{-1}\circ(\mathcal{J}_{X}\circ\varphi^{-1})^{tt}\circ(\mathcal{J}_{Y}^{t})^{-1}\\
&=\mathcal{J}_{(X_{b}')_{b}'}^{-1}\circ(\mathcal{J}_{X}\circ\varphi^{-1})^{tt}\circ\mathcal{J}_{Y_{b}'}
 =\mathcal{J}_{(X_{b}')_{b}'}^{-1}\circ\mathcal{J}_{(X_{b}')_{b}'}\circ(\mathcal{J}_{X}\circ\varphi^{-1})
 =\mathcal{J}_{X}\circ\varphi^{-1}.
\end{align*}
We conclude that $(Y,\varphi)\sim( X_{b}',\mathcal{J}_{X})$. Due to \prettyref{prop:strongly_unique_equivalent} this means that 
$X$ has a strongly unique complete barrelled predual. 
\end{proof}

If $X$ is a reflexive Fr\'echet space, then $X_{b}'$ is a complete reflexive DF-space by \cite[12.4.5 Theorem, p.~260]{jarchow1981} 
and so $X$ has a strongly unique complete barrelled DF-predual by \prettyref{rem:strongly_unique}. 
Similarly, if $X$ is a reflexive bornological DF-space, then $X_{b}'$ is a complete reflexive Fr\'echet space 
by \cite[12.4.2 Theorem, p.~258]{jarchow1981} and so $X$ has a strongly unique Fr\'echet predual. 
If $X$ is a reflexive Banach space, then $X_{b}'$ is a completely normable reflexive space and so 
$X$ has a strongly unique Banach predual. Let us turn to linearisations of function spaces whose definition is motivated by 
the notion of strong Banach linearisations \cite[p.~184, 187]{jaramillo2009}.

\section{Linearisation and uniqueness}
\label{sect:linearisation}

We begin this section with the definition of a linearisation. 

\begin{defn}[{\cite[Definition 2.3, p.~1593]{kruse2023a}}]\label{defn:linearisation}
Let $\F$ be a linear space of $\K$-valued functions on a non-empty set $\Omega$. 
\begin{enumerate}
\item[(a)] We call a triple $(\delta,Y,T)$ of a locally convex Hausdorff space $Y$ over the field $\K$, 
a map $\delta\colon\Omega\to Y$ and an algebraic isomorphism $T\colon\F\to Y'$ a \emph{linearisation of} $\F$ 
if $T(f)\circ \delta= f$ for all $f\in\F$. 
\item[(b)] Let $\Omega$ be a topological Hausdorff space. We call a linearisation $(\delta,Y,T)$ of $\F$ \emph{continuous} 
if $\delta$ is continuous.
\item[(c)] Let $\F$ be a locally convex Hausdorff space. We call a linearisation $(\delta,Y,T)$ of $\F$ \emph{strong} if $T\colon\F\to Y_{b}'$ is a topological isomorphism. 
\item[(d)] We call a (strong) linearisation $(\delta,Y,T)$ of $\F$ a (strong) \emph{complete barrelled (Fr\'echet, DF-, Banach) linearisation} if $Y$ is a complete barrelled (Fr\'echet, DF-, completely normable) space.
\item[(e)] We say that $\F$ \emph{admits a (continuous, strong, complete barrelled, Fr\'echet, DF-, Banach) linearisation} 
if there exists a (continuous, strong, complete barrelled, Fr\'echet, DF-, Banach) linearisation $(\delta,Y,T)$ 
of $\F$.
\end{enumerate}
\end{defn}

Clearly, the tuple $(Y,T)$ of a linearisation $(\delta,Y,T)$ is a predual of $\F$. 
If we have two strong linearisations of a common function space $\F$ such that the corresponding preduals are equivalent, 
then we have the following relation between the $\delta$-maps.

\begin{prop}\label{prop:lin_relation_deltas}
Let $\F$ be a locally convex Hausdorff space of $\K$-valued functions on a non-empty set $\Omega$ and 
$(\delta,Y,T)$ and $(\widetilde{\delta},Z,\varphi)$ strong linearisations of $\F$. 
If there is a topological isomorphism $\lambda\colon Y\to Z$ such that $\lambda^{t}=T\circ\varphi^{-1}$, then it holds 
$\lambda(\delta(x))=\widetilde{\delta}(x)$ for all $x\in\Omega$. 
\end{prop}
\begin{proof}
We note that 
\begin{align*}
 z'(\lambda(\delta(x)))
&=\lambda^{t}(z')(\delta(x))
 =(T\circ\varphi^{-1})(z')(\delta(x))
 =T(\varphi^{-1}(z'))(\delta(x))\\
&=\varphi^{-1}(z')(x)
 =\varphi(\varphi^{-1}(z'))(\widetilde{\delta}(x))
 =z'(\widetilde{\delta}(x))
\end{align*}
for all $z'\in Z'$ and $x\in\Omega$. Since $Z$ is Hausdorff, our statement follows from the Hahn--Banach theorem. 
\end{proof}

Further, we observe that the map $\delta$ which makes a predual a strong linearisation is unique. 

\begin{rem}\label{rem:lin_unique_delta}
If $\F$ is a locally convex Hausdorff space of $\K$-valued functions on a non-empty set $\Omega$ and 
$(\delta,Y,T)$ and $(\widetilde{\delta},Y,T)$ are strong linearisations of $\F$, then it holds $\delta=\widetilde{\delta}$. Indeed, this follows from \prettyref{prop:lin_relation_deltas} with 
$Z\coloneqq Y$, $\varphi\coloneqq T$ and $\lambda\coloneqq\id$. 
\end{rem}

Now, we recall two conditions, named $\operatorname{(BBC)}$ and $\operatorname{(CNC)}$, that we need to guarantee the existence of a strong complete barrelled linearisation of a bornological function space (see \cite[p.~114]{bierstedt1992}, 
\cite[2.~Corollary, p.~115]{bierstedt1992} and \cite[Theorem 1, p.~320--321]{mujica1984a}). 
 
\begin{defn}[{\cite[Definition 3.1, p.~1596--1597]{kruse2023a}}]
Let $(X,\tau)$ be a locally convex Hausdorff space. 
\begin{enumerate}
\item[(a)] We say that $(X,\tau)$ satisfies condition $\operatorname{(BBC)}$ 
if there exists a locally convex Hausdorff topology $\widetilde{\tau}$ on $X$ such that every $\tau$-bounded subset of $X$ 
is contained in an absolutely convex $\tau$-bounded $\widetilde{\tau}$-compact set. 
\item[(b)] We say that $(X,\tau)$ satisfies condition $\operatorname{(CNC)}$ if there exists a locally convex Hausdorff topology 
$\widetilde{\tau}$ on $X$ such that $\tau$ has a $0$-neighbourhood basis $\mathcal{U}_{0}$ 
of absolutely convex $\widetilde{\tau}$-closed sets. 
\end{enumerate}
If we want to emphasize the dependency on $\widetilde{\tau}$ we say that $(X,\tau)$ satifies $\operatorname{(BBC)}$ 
resp.~$\operatorname{(CNC)}$ for $\widetilde{\tau}$. We say $(X,\tau)$ satisfies $\operatorname{(BBC)}$ and 
$\operatorname{(CNC)}$ for $\widetilde{\tau}$ if it satisfies both conditions for the same $\widetilde{\tau}$.
\end{defn}

In order to obtain a candidate for the predual $(Y,T)$ of a given function space $\F$, we recall the following space of linear functionals from \cite{bierstedt1992,kruse2023a,mujica1984a,ng1971}. 
Let $(X,\tau)$ be a locally convex Hausdorff space, $\mathcal{B}$ the family of $\tau$-bounded sets and $\widetilde{\tau}$ 
another locally convex Hausdorff topology on $X$. We denote by $X^{\star}$ the algebraic dual space of $X$ and define 
\[
 X_{\mathcal{B},\widetilde{\tau}}'
\coloneqq\{x^{\star}\in X^{\star}\;|\;x^{\star}_{\mid B}\text{ is }\widetilde{\tau}\text{-continuous for all }B\in\mathcal{B}\}
\]
and observe that $(X,\widetilde{\tau})'\subset X_{\mathcal{B},\widetilde{\tau}}'$ as linear spaces. 
We equip $X_{\mathcal{B},\widetilde{\tau}}'$ with the topology 
$\beta\coloneqq\beta(X_{\mathcal{B},\widetilde{\tau}}',(X,\tau))$ 
of uniform convergence on the $\tau$-bounded subsets of $X$. 

\begin{prop}[{\cite[Proposition 3.7, p.~1599]{kruse2023a}}]\label{prop:predual_complete}
Let $(X,\tau)$ be a bornological locally convex Hausdorff space satisfying $\operatorname{(BBC)}$ for some $\widetilde{\tau}$ 
and $\mathcal{B}$ the family of $\tau$-bounded sets. Then the following assertions hold.
\begin{enumerate}
\item[(a)] $X_{\mathcal{B},\widetilde{\tau}}'$ is a closed subspace of the complete space $(X,\tau)_{b}'$. 
In particular, $(X_{\mathcal{B},\widetilde{\tau}}',\beta)$ is complete. 
\item[(b)] If $(X,\tau)$ is a DF-space, then $(X_{\mathcal{B},\widetilde{\tau}}',\beta)$ is 
a Fr\'echet space. 
\item[(c)] If $(X,\tau)$ is normable, then $(X_{\mathcal{B},\widetilde{\tau}}',\beta)$ is 
completely normable. 
\end{enumerate}
\end{prop}

We recall from \cite[p.~1611--1612]{kruse2023a} that a topological space $\Omega$ is said to be a $gk_{\R}$\emph{-space} 
if for any completely regular space $Y$ and any map $f\colon\Omega\to Y$, 
whose restriction to each compact $K\subset\Omega$ is continuous, the map is already continuous on 
$\Omega$. Examples of Hausdorff $gk_{\R}$-spaces are Hausdorff $k$-spaces, metrisable spaces, 
locally compact Hausdorff spaces and strong duals of Fr\'echet--Montel spaces (\emph{DFM-spaces}). 
Further, for a locally convex Hausdorff space $\F$ of $\K$-valued functions on a non-empty set $\Omega$ 
we define $\delta_{x}\colon\F\to\K$, $\delta_{x}(f)\coloneqq f(x)$, for $x\in\Omega$. 

\begin{cor}[{\cite[Corollary 4.1, Remark 4.3, Theorem 4.10, p.~1609, 1612]{kruse2023a}}]\label{cor:scb_linearisation}
Let $(\F,\tau)$ be a bornological locally convex Hausdorff space of $\K$-valued functions on a non-empty set $\Omega$ satisfying 
$\operatorname{(BBC)}$ and $\operatorname{(CNC)}$ for some $\widetilde{\tau}\leq \tau_{\operatorname{p}}$ and $\mathcal{B}$ 
the family of $\tau$-bounded sets. Then the following assertions hold.
\begin{enumerate}
\item[(a)] $\Delta(x)\coloneqq\delta_{x}\in\F_{\mathcal{B},\widetilde{\tau}}'$ for all $x\in\Omega$ and 
$(\Delta,\F_{\mathcal{B},\widetilde{\tau}}',\mathcal{I})$ is a strong complete barrelled linearisation of $\F$ 
where 
\[
\mathcal{I}\colon (\F,\tau)\to (\F_{\mathcal{B},\widetilde{\tau}}',\beta)_{b}',\;
f\longmapsto [f'\mapsto f'(f)].
\]
\item[(b)] If $\Omega$ is a $gk_{\R}$-space and $\F$ a space of continuous functions, then the map
$\Delta\colon \Omega\to (\F_{\mathcal{B},\widetilde{\tau}}',\beta)$ is continuous. 
\end{enumerate}
\end{cor}

We note that the conditions $\operatorname{(BBC)}$ and $\operatorname{(CNC)}$ for some $\widetilde{\tau}\leq \tau_{\operatorname{p}}$ are also necessary for the existence of a strong complete barrelled linearisation of a bornological 
function space $\F$ (see \cite[Theorem 4.5, p.~1609]{kruse2023a}). 

Let us turn to some examples from \cite{{kruse2023a}}. Let $\Omega$ be a non-empty topological Hausdorff space. 
We call $\mathcal{V}$ a \emph{directed family of continuous weights} 
if $\mathcal{V}$ is a family of continuous functions $v\colon\Omega\to[0,\infty)$ such that for every $v_{1},v_{2}\in \mathcal{V}$ there are 
$C\geq 0$ and $v_{0}\in \mathcal{V}$ with $\max(v_{1},v_{2})\leq Cv_{0}$ on $\Omega$. 
We call a directed family of continuous weights $\mathcal{V}$ \emph{point-detecting} if for every $x\in\Omega$ there is $v\in \mathcal{V}$ 
such that $v(x)>0$. For an open set $\Omega\subset\R^{d}$ we denote by $\mathcal{C}^{\infty}(\Omega)$ the space of $\K$-valued infinitely continuously partially differentiable functions on $\Omega$. 
The next examples are slight generalisations of \cite[p.~34]{bonet2002} and \cite[3.~Examples B, p.~125--126]{bierstedt1992} where the weighted spaces 
$\mathcal{HV}(\Omega)$ and $\mathcal{VH}(\Omega)$ of holomorphic functions on an open connected set $\Omega\subset\C^{d}$ are considered and $\mathcal{V}$ is a point-detecting Nachbin family of continuous weights. 
Using \prettyref{prop:predual_complete} and \prettyref{cor:scb_linearisation}, 
we get the following examples of continuous strong complete barrelled (Fr\'echet, DF-, Banach) linearisations.

\begin{exa}[{\cite[Examples 3.3, 3.25, 4.11, p.~1597--1598, 1607--1608, 1613]{kruse2023a}}]\label{ex:cont_strong_lin}
Let $\Omega\subset\R^{d}$ be open and $P(\partial)$ a hypoelliptic linear partial differential operator on $\mathcal{C}^{\infty}(\Omega)$. 

(i) Let $\mathcal{V}$ be a point-detecting directed family of continuous weights. We define the space
\[
\mathcal{C}_{P}\mathcal{V}(\Omega)\coloneqq\{f\in\mathcal{C}_{P}(\Omega)\;|\;\forall\;v\in \mathcal{V}:\;\|f\|_{v}\coloneqq\sup_{x\in\Omega}|f(x)|v(x)<\infty\},
\] 
where $\mathcal{C}_{P}(\Omega)\coloneqq\{f\in\mathcal{C}^{\infty}(\Omega)\;|\;f\in\ker{P(\partial)}\}$, and equip 
$\mathcal{C}_{P}\mathcal{V}(\Omega)$ with the locally convex Hausdorff topology $\tau_{\mathcal{V}}$ induced by the seminorms $(\|\cdot\|_{v})_{v\in \mathcal{V}}$. 
If the space $(\mathcal{C}_{P}\mathcal{V}(\Omega),\tau_{\mathcal{V}})$ is bornological, then $(\Delta,\mathcal{C}_{P}\mathcal{V}(\Omega)_{\mathcal{B},\tau_{\operatorname{co}}}',\mathcal{I})$ is a continuous strong complete barrelled 
linearisation of $\mathcal{C}_{P}\mathcal{V}(\Omega)$.
If $\mathcal{V}$ is countable and \emph{increasing}, i.e.~$v_{n}\leq v_{n+1}$ for all $n\in\N$, then $(\Delta,\mathcal{C}_{P}\mathcal{V}(\Omega)_{\mathcal{B},\tau_{\operatorname{co}}}',\mathcal{I})$ is a continuous strong complete barrelled DF-linearisation of 
$\mathcal{C}_{P}\mathcal{V}(\Omega)$, and if $\mathcal{V}=\{v\}$, then $(\Delta,\mathcal{C}_{P}v(\Omega)_{\mathcal{B},\tau_{\operatorname{co}}}',\mathcal{I})$ is a continuous strong Banach linearisation of 
$\mathcal{C}_{P}v(\Omega)$. 

(ii) Let $\mathcal{V}\coloneqq (v_{n})_{n\in\N}$ be a \emph{drecreasing}, i.e.~$v_{n+1}\leq v_{n}$ for all $n\in\N$, family of continuous functions $v_{n}\colon\Omega\to(0,\infty)$. In addition, let $\mathcal{V}$ be \emph{regularly decreasing}, 
i.e.~for every $n\in\N$ there is $m\geq n$ such that for every $U\subset\Omega$ with 
$\inf_{x\in U}v_{m}(x)/v_{n}(x)>0$ we also have $\inf_{x\in U}v_{k}(x)/v_{n}(x)>0$ for all $k\geq m+1$. 
We define the inductive limit
\[
\mathcal{VC}_{P}(\Omega)\coloneqq\lim\limits_{\substack{\longrightarrow\\n\in \N}}\,\mathcal{C}_{P}v_{n}(\Omega)
\] 
of the Banach spaces $(\mathcal{C}_{P}v_{n}(\Omega),\|\cdot\|_{v_{n}})$, and equip 
$\mathcal{VC}_{P}(\Omega)$ with its locally convex inductive limit topology ${_{\mathcal{V}}\tau}$. 
Then $(\mathcal{VC}_{P}(\Omega),{_{\mathcal{V}}\tau})$ is a (ultra)bornological Hausdorff DF-space and 
$(\Delta,\mathcal{VC}_{P}(\Omega)_{\mathcal{B},\tau_{\operatorname{co}}}',\mathcal{I})$ is 
a continuous strong Fr\'echet linearisation of $\mathcal{VC}_{P}(\Omega)$.
\end{exa}

Now, we turn to the relation between strong linearisations and (strong) uniqueness of preduals. 
Let $X$ be a dual space with predual $(Y,\varphi)$. 
Considering $X$ as the dual space of $Y$, we define the system of seminorms 
\[
p_{N}(x)\coloneqq\sup_{y\in N}|\varphi(x)(y)|,\quad x\in X,
\]
for finite sets $N\subset Y$, which induces a locally convex Hausdorff topology on $X$ 
w.r.t.~the dual paring $\langle X,Y,\varphi\rangle$ and we denote this topology by 
$\sigma_{\varphi}(X,Y)$.

\begin{prop}\label{prop:strongly_unique_C_predual}
Let $\F$ be a locally convex Hausdorff space of $\K$-valued functions on a non-empty set $\Omega$, 
$\mathcal{C}$ be a subclass of the class of complete barrelled locally convex Hausdorff 
spaces such that $\mathcal{C}$ is closed under topological isomorphisms, and $(\delta,Y,T)$ a strong linearisation of $\F$ such that $Y\in\mathcal{C}$. Consider the following assertions.
\begin{enumerate}
\item[(a)] $\F$ has a strongly unique $\mathcal{C}$ predual. 
\item[(b)] For every predual $(Z,\varphi)$ of $\F$ such that $Z\in\mathcal{C}$ and every $x\in\Omega$ there is a (unique) $z_{x}\in Z$ with 
$T(\cdot)(\delta(x))=\varphi(\cdot)(z_{x})$.
\item[(c)] For every predual $(Z,\varphi)$ of $\F$ such that $Z\in\mathcal{C}$ and every $x\in\Omega$ it holds
$T(\cdot)(\delta(x))\in(\F,\sigma_{\varphi}(\F,Z))'$.
\item[(d)] For every predual $(Z,\varphi)$ of $\F$ such that $Z\in\mathcal{C}$ there is a (unique) map 
$\widetilde{\delta}\colon\Omega\to Z$ such that $(\widetilde{\delta},Z,\varphi)$ is a strong linearisation of $\F$.
\end{enumerate}
Then it holds (a)$\Rightarrow$(b)$\Leftrightarrow$(c)$\Leftrightarrow$(d). If the family $(\delta(x))_{x\in\Omega}$ is linearly independent, then it holds (b)$\Rightarrow$(a). 
\end{prop}
\begin{proof}
(b)$\Leftrightarrow$(c) This equivalence follows directly from \cite[Chap.~IV, \S1, 1.2, p.~124]{schaefer1971}.

(a)$\Rightarrow$(b) Let $(Z,\varphi)$ be predual of $\F$ such that $Z\in\mathcal{C}$. 
Since $\F$ has a strongly unique $\mathcal{C}$ predual, there is a topological isomorphism $\lambda\colon Y\to Z$ such 
that $\lambda^{t}=T\circ\varphi^{-1}$ by \prettyref{prop:strongly_unique_equivalent}. For $x\in\Omega$ we set 
$z_{x}\coloneqq \lambda(\delta(x))\in Z$ and observe that 
\[
 \varphi(f)(z_{x})
=\varphi(f)(\lambda(\delta(x)))
=\lambda^{t}(\varphi(f))(\delta(x))
=(T\circ\varphi^{-1})(\varphi(f))(\delta(x))
=T(f)(\delta(x))
\]
for all $f\in\F$. The existence of such a $z_{x}$ already implies that it is unique by \cite[Chap.~IV, \S1, 1.2, p.~124]{schaefer1971}.

(b)$\Rightarrow$(d) Let $(Z,\varphi)$ be a predual of $\F$ such that $Z\in\mathcal{C}$. We set $\widetilde{\delta}\colon\Omega\to Z$, 
$\widetilde{\delta}(x)\coloneqq z_{x}$. Then we have $\varphi(f)(\widetilde{\delta}(x))=T(f)(\delta(x))=f(x)$ for all $f\in\F$, 
which implies that $(\widetilde{\delta},Z,\varphi)$ is a strong linearisation of $\F$. The uniqueness of $\widetilde{\delta}$ follows from \prettyref{rem:lin_unique_delta}.

(d)$\Rightarrow$(b) This statement follows from the choice $z_{x}\coloneqq\widetilde{\delta}(x)\in Z$ for $x\in\Omega$. 

(b)$\Rightarrow$(a) if the family $(\delta(x))_{x\in\Omega}$ is linearly independent. 
Let $(Z,\varphi)$ be a predual of $\F$ such that $Z\in\mathcal{C}$. 
We set $Y_{0}\coloneqq\operatorname{span}\{\delta(x)\;|\;x\in\Omega\}$ and $Z_{0}\coloneqq\operatorname{span}\{z_{x}\;|\;x\in\Omega\}$ 
and note that the linear map $\lambda_{0}\colon Y_{0}\to Z_{0}$ induced by $\lambda_{0}(\delta(x))\coloneqq z_{x}$ for $x\in\Omega$ 
is well-defined since $(\delta(x))_{x\in\Omega}$ is linearly independent. Clearly, $\lambda_{0}$ is surjective. Next, we show that it is 
also injective. Let $y\in Y_{0}$ such that $\lambda_{0}(y)=0$. Then $y$ can be represented as $y=\sum_{x\in\Omega}a_{x}\delta(x)$ 
with finitely many non-zero $a_{x}\in\K$. We note that 
\[
 0
=\lambda_{0}(y)
=\lambda_{0}(\sum_{x\in\Omega}a_{x}\delta(x))
=\sum_{x\in\Omega}a_{x}z_{x}
\]
and 
\[
 0
=\varphi(f)(\sum_{x\in\Omega}a_{x}z_{x})
=\sum_{x\in\Omega}a_{x}\varphi(f)(z_{x})
=T(f)(\sum_{x\in\Omega}a_{x}\delta(x))
=\Phi_{T}(y)(f)
\]
for all $f\in\F$. Due to the injectivity of $\Phi_{T}$ by \prettyref{prop:predual_into_dual} we get that 
$y=0$. Thus $\lambda_{0}$ is injective. 

We claim that $\lambda_{0}$ is a topological isomorphism. Indeed, for every $x\in\Omega$ we have 
\[
\Phi_{\varphi}(z_{x})=\varphi(\cdot)(z_{x})=T(\cdot)(\delta(x))=\Phi_{T}(\delta(x))
\]
and hence $\lambda_{0}(\delta(x))=(\Phi_{\varphi}^{-1}\circ\Phi_{T})(\delta(x))$, implying our claim as 
$(\Phi_{\varphi}^{-1}\circ\Phi_{T})_{\mid Y_{0}}\colon Y_{0}\to Z_{0}$ is a topological isomorphism by 
\prettyref{prop:predual_into_dual}. Due to \cite[3.4.2 Theorem, p.~61--62]{jarchow1981} and the density of 
$Y_{0}$ in the complete space $Y$ by \cite[Proposition 2.7, p.~1596]{kruse2023a} we can uniquely extend 
$\lambda_{0}$ to a topological isomorphism $\lambda\colon Y \to \overline{Z_{0}}$ where $\overline{Z_{0}}$ denotes 
the closure of $Z_{0}$ in the complete space $Z$. 

Next, we show that $\overline{Z_{0}}=Z$. Suppose there is some $z\in Z$ such that $z\notin\overline{Z_{0}}$. 
Since $\overline{Z_{0}}$ is closed in $Z$, there is $z'\in Z'$ such that $z'_{\mid \overline{Z_{0}}}=0$ and $z'(z)=1$ by 
the Hahn--Banach theorem. It follows that $0\neq(T\circ\varphi^{-1})(z')\in Y'$ and 
\[
 (T\circ\varphi^{-1})(z')(\delta(x))
=T(\varphi^{-1}(z'))(\delta(x))
=\varphi(\varphi^{-1}(z'))(z_{x})
=z'(z_{x})
=0
\]
for all $x\in\Omega$. The density of $Y_{0}$ in $Y$ implies that $(T\circ\varphi^{-1})(z')=0$, which is a contradiction.

Therefore $\lambda\colon Y \to Z$ is a topological isomorphism. Further, we have for every $z'\in Z'$ with 
$y'\coloneqq (T\circ\varphi^{-1})(z')\in Y'$ that 
\begin{align*}
  \lambda^{t}(z')(\delta(x))
&=z'(\lambda(\delta(x)))
 =z'(\lambda_{0}(\delta(x)))
 =z'(z_{x})
 =(\varphi\circ T^{-1})(y')(z_{x})\\
&=\varphi(T^{-1}(y'))(z_{x})
 =T(T^{-1}(y'))(\delta(x))
 =y'(\delta(x))\\
&=(T\circ\varphi^{-1})(z')(\delta(x))
\end{align*}
for all $x\in\Omega$. Again, the density of $Y_{0}$ in $Y$ implies that $\lambda^{t}=T\circ\varphi^{-1}$, which means 
that $(Z,\varphi)\sim (Y,T)$. We deduce from \prettyref{prop:strongly_unique_equivalent} that $\F$ has a strongly unique $\mathcal{C}$ predual.
\end{proof}

\begin{cor}\label{cor:unique_C_predual}
Let $\F$ be a locally convex Hausdorff space of $\K$-valued functions on a non-empty set $\Omega$, 
$\mathcal{C}$ be a subclass of the class of complete barrelled locally convex Hausdorff 
spaces such that $\mathcal{C}$ is closed under topological isomorphisms, and $(\delta,Y,T)$ a strong linearisation of $\F$ such that $Y\in\mathcal{C}$. Consider the following assertions.
\begin{enumerate}
\item[(a)] $\F$ has a unique $\mathcal{C}$ predual. 
\item[(b)] For every predual $(Z,\varphi)$ of $\F$ such that $Z\in\mathcal{C}$ there is a topological isomorphism 
$\psi\colon\F\to Z_{b}'$ such that for every $x\in\Omega$ there is a (unique) $z_{x}\in Z$ with 
$T(\cdot)(\delta(x))=\psi(\cdot)(z_{x})$.
\item[(c)] For every predual $(Z,\varphi)$ of $\F$ such that $Z\in\mathcal{C}$ there is a topological isomorphism 
$\psi\colon\F\to Z_{b}'$ such that for every $x\in\Omega$ it holds $T(\cdot)(\delta(x))\in(\F,\sigma_{\psi}(\F,Z))'$.
\item[(d)] For every predual $(Z,\varphi)$ of $\F$ such that $Z\in\mathcal{C}$ there is a topological isomorphism 
$\psi\colon\F\to Z_{b}'$ and a (unique) map $\widetilde{\delta}\colon\Omega\to Z$ such that $(\widetilde{\delta},Z,\psi)$ 
is a strong linearisation of $\F$.
\end{enumerate}
Then it holds (a)$\Rightarrow$(b)$\Leftrightarrow$(c)$\Leftrightarrow$(d). If the family $(\delta(x))_{x\in\Omega}$ is linearly independent, then it holds (b)$\Rightarrow$(a). 
\end{cor}
\begin{proof}
This statement follows from the proof of \prettyref{prop:strongly_unique_C_predual} with $\varphi$ replaced by $\psi$ and 
\prettyref{prop:strongly_unique_equivalent} replaced by \prettyref{prop:unique_equivalent}.
\end{proof}

Looking at statement (b) of \prettyref{prop:strongly_unique_C_predual} and \prettyref{cor:unique_C_predual}, 
we also see that continuous point evaluation functionals naturally appear in the context of 
strong complete barrelled linearisations.

\begin{rem}\label{rem:point_eval_scb_lin}
Let $\F$ be a locally convex Hausdorff space of $\K$-valued functions on a non-empty set $\Omega$, 
$(\delta,Y,T)$ a strong linearisation of $\F$ such that $Y$ is quasi-barrelled. 
Then we have $\delta_{x}=T(\cdot)(\delta(x))$ and $\delta_{x}\in\F'$ for every $x\in\Omega$ 
by \prettyref{prop:predual_into_dual} since $T(\cdot)(\delta(x))=\Phi_{T}(\delta(x))$. 
\end{rem}

In our last section we will show in \prettyref{cor:strongly_unique_C_predual_without_lin_ind} and \prettyref{cor:unique_C_predual_without_lin_ind} how the condition of linear independence 
for the implication (b)$\Rightarrow$(a) in \prettyref{prop:strongly_unique_C_predual} 
and \prettyref{cor:unique_C_predual} can be avoided. 

\begin{rem}\label{rem:lin_indep}
Let $\F$ be a linear space of $\K$-valued functions on a non-empty set 
$\Omega$ such that for every finite set $A\subset\Omega$ and $x\in A$ there is $f\in\F$ such that $f(x)\neq 0$ and $f(z)=0$ for all $z\in A\setminus \{x\}$. If $(\delta,Y,T)$ is a linearisation of $\F$, then the family $(\delta(x))_{x\in\Omega}$ is linearly independent. Indeed, for a family $(a_{x})_{x\in\Omega}$ in $\K$ with finitely many 
non-zero $a_{x}$ we have
\[
T(f)(\sum_{x\in\Omega}a_{x}\delta(x))=\sum_{x\in\Omega}a_{x}f(x)
\]
for all $f\in\F$, which implies our claim.  
\end{rem}

The condition on $\F$ in \prettyref{rem:lin_indep} is for example fulfilled if $\Omega\subset\C$ and $\F$ contains all polynomials of degree less than $|\Omega|$ where $|\Omega|$ denotes the cardinality of $\Omega$.

\section{Linearisation of vector-valued functions}
\label{sect:linearisation_vetor_valued}

In this section we derive a strong linearisation of weak vector-valued functions from a strong linearisation of 
scalar-valued functions. Let $\F$ be a locally convex Hausdorff space of $\K$-valued functions on a non-empty set $\Omega$ 
with fundamental system of seminorms $\Gamma_{\F}$. 
For a locally convex Hausdorff space $E$ over the field $\K$ with fundamental system of seminorms $\Gamma_{E}$ 
we define the \emph{space of weak $E$-valued $\mathcal{F}$-functions} by 
\[
\FE_{\sigma}\coloneqq\{f\colon \Omega\to E\;|\;\forall\;e'\in E':\;e'\circ f\in\F\}.
\]
For $p\in\Gamma_{E}$ we set $U_{p}\coloneqq\{x\in E\;|\; p(x)<1\}$ and recall that $U_{p}^{\circ}$ denotes the polar set of $U_{p}$ w.r.t.~the dual pairing $\langle E,E'\rangle$. 
We define the linear subspace of $\FE_{\sigma}$ given by 
\[
\FE_{\sigma,b}\coloneqq\{f\in\FE_{\sigma}\;|\;\forall\;q\in\Gamma_{\F},\,p\in\Gamma_{E}:\;
\|f\|_{\sigma,q,p}\coloneqq \sup_{e'\in U_{p}^{\circ}}q(e'\circ f)<\infty\}.
\]
$\FE_{\sigma,b}$ equipped with the system of seminorms 
$(\|\cdot\|_{\sigma,q,p})_{q\in\Gamma_{\F},p\in\Gamma_{E}}$ becomes a locally convex Hausdorff space. 
We note that $\mathcal{F}(\Omega,\K)_{\sigma,b}=\mathcal{F}(\Omega,\K)_{\sigma}=\F$. 
Further, for a large class of spaces $\F$, namely BC-spaces such that the point evaluation functionals belong to the dual, 
the spaces $\FE_{\sigma}$ and $\FE_{\sigma,b}$ actually coincide 
for any $E$. Let us recall the definition of a BC-space from \cite[p.~395]{powell1975}.
A locally convex Hausdorff space $X$ is called a \emph{BC-space} 
if for every Banach space $Y$ and every linear map $f\colon Y\to X$ 
with closed graph in $Y\times X$, one has that $f$ is continuous. 
A characterisation of BC-spaces is given in \cite[6.1 Corollary, p.~400--401]{powell1975}. 
Since every Banach space is ultrabornological and barrelled, 
the \cite[Closed graph theorem 24.31, p.~289]{meisevogt1997} of de Wilde 
and the Pt\'{a}k--K\={o}mura--Adasch--Valdivia closed graph theorem \cite[\S34, 9.(7), p.~46]{koethe1979} 
imply that webbed spaces and $B_{r}$-complete spaces are BC-spaces. In particular, 
the Fr\'echet space $\mathcal{C}_{P}\mathcal{V}(\Omega)$ for a countable point-detecting increasing family 
$\mathcal{V}$ and the 
complete DF-space $\mathcal{VC}_{P}(\Omega)$ for a countable decreasing family $\mathcal{V}$ 
from \prettyref{ex:cont_strong_lin} are webbed and thus BC-spaces.

\begin{rem}\label{rem:BC_space}
If $\F$ is a BC-space of $\K$-valued functions on a non-empty set $\Omega$ such that $\delta_{x}\in\F'$ 
for all $x\in\Omega$ and $E$ is a locally convex Hausdorff space over the 
field $\K$, then $\FE_{\sigma}=\FE_{\sigma,b}$ by \cite[3.18 Proposition, p.~319]{kruse2018_3}. 
\end{rem}

\begin{rem}\label{rem:cond_i_linearisation}
Let $\F$ be a locally convex Hausdorff space of $\K$-valued functions on a non-empty set $\Omega$, $Y$ a locally convex Hausdorff space 
over the field $\K$ and $\delta\colon\Omega\to Y$. Then condition (i) of \cite[Proposition 2.5, p.~1595]{kruse2023a} 
is equivalent to $\delta\in\mathcal{F}(\Omega,Y)_{\sigma}$.
\end{rem}

For two locally convex Hausdorff spaces $Y$ and $E$ over the field $\K$ we denote by $L(Y,E)$ the space of continuous linear maps 
from $Y$ to $E$. Moreover, we write $L_{b}(Y,E)$ if the space $L(Y,E)$ is equipped with the topology of uniform convergence on 
bounded subsets of $Y$.  

\begin{prop}\label{prop:linearisation_into}
Let $\F$ be a locally convex Hausdorff space of $\K$-valued functions on a non-empty set $\Omega$, 
$(\delta,Y,T)$ a strong linearisation of $\F$ and $E$ a locally convex Hausdorff space over the field $\K$ 
with fundamental system of seminorms $\Gamma_{E}$. Then the following assertions hold. 
\begin{enumerate}
\item[(a)] The map 
\[
\chi\colon L_{b}(Y,E)\to\FE_{\sigma,b},\;\chi(u)\coloneqq u\circ \delta,
\]
is well-defined, linear, continuous and injective.
\item[(b)] If $(\F,\|\cdot\|)$ is a Banach space such that $\delta_{x}\in (\F,\|\cdot\|)'$ for all $x\in\Omega$, 
$(Y,\|\cdot\|_{Y})$ a normed space and the map
$T\colon (\F,\|\cdot\|)\to ((Y,\|\cdot\|_{Y})',\|\cdot\|_{Y'})$ an isometry, 
then $\FE_{\sigma}=\FE_{\sigma,b}$ and $\chi$ is a topological isomorphism into and for all $p\in\Gamma_{E}$
\[
\|\chi(u)\|_{\sigma,p}=\sup_{\substack{y\in Y\\ \|y\|_{Y}\leq 1}}p(u(y)),\quad u\in L(Y,E).
\]
\end{enumerate}
\end{prop}
\begin{proof}
(a) We note that $e'\circ u\in Y'$ for all $e'\in E'$ and $u\in L(Y,E)$. 
Therefore $e'\circ \chi(u)=(e'\circ u)\circ\delta\in\F$ for all 
$e'\in E'$ and $u\in L(Y,E)$, which yields that $\chi(u)\in\FE_{\sigma}$. 

Let $q\in\Gamma_{\F}$ and  $p\in\Gamma_{E}$ where $\Gamma_{\F}$ is a fundamental system of seminorms of $\F$. 
For all $u\in L(Y,E)$ and $e'\in E'$ there is $f_{e'\circ u}\in\F$ such that $e'\circ u=T(f_{e'\circ u})$ and 
$T(f_{e'\circ u})\circ\delta=f_{e'\circ u}$ because $(\delta,Y,T)$ is a linearisation of $\F$ and $e'\circ u\in Y'$.
Due to the continuity of $T^{-1}\colon Y_{b}'\to \F$ there are $C\geq 0$ and a bounded set $B\subset Y$ such that
\begin{align}\label{eq:isometry_vector_valued}
  \|\chi(u)\|_{\sigma,q,p}
&=\sup_{e'\in U_{p}^{\circ}}q(e'\circ u\circ\delta)
 =\sup_{e'\in U_{p}^{\circ}}q(T(f_{e'\circ u})\circ\delta)
 =\sup_{e'\in U_{p}^{\circ}}q(f_{e'\circ u})\nonumber\\
&=\sup_{e'\in U_{p}^{\circ}}q(T^{-1}(T(f_{e'\circ u})))
 \leq C \sup_{e'\in U_{p}^{\circ}}\sup_{y\in B}|T(f_{e'\circ u})(y)|\\
&=C \sup_{e'\in U_{p}^{\circ}}\sup_{y\in B}|(e'\circ u)(y)|
 =C \sup_{y\in B}p(u(y))<\infty\nonumber
\end{align}
for all $u\in L(Y,E)$ where we used \cite[Proposition 22.14, p.~256]{meisevogt1997} for the last equation. 
Therefore $\chi(u)\in\FE_{\sigma,b}$, which means that the map $\chi$ is well-defined. The map $\chi$ is also linear and 
thus the estimate above implies that it is continuous. The map $\chi$ is also injective as the span of 
$\{\delta(x)\;|\;x\in\Omega\}$ is dense in $Y$ by \cite[Proposition 2.7, p.~1596]{kruse2023a}. 

(b) We have $\FE_{\sigma}=\FE_{\sigma,b}$ by \prettyref{rem:BC_space} since Banach spaces are webbed and so BC-spaces. 
The rest of part (b) follows from part (a) since we have equality in \eqref{eq:isometry_vector_valued} with $q=\|\cdot\|$, $C=1$ and 
$B=B_{\|\cdot\|_{Y}}$.
\end{proof}

\begin{rem}\label{rem:linearisation_isometry_Banach_valued}
Let the assumptions of \prettyref{prop:linearisation_into} (b) be fulfilled. 
If $(E,\|\cdot\|_{E})$ is also a normed space and $L(Y,E)$ equipped with operator norm given by
$\|u\|_{L(Y,E)}\coloneqq\sup\{\|u(y)\|_{E}\;|\;y\in Y,\,\|y\|_{Y}\leq 1\}$, $u\in L(Y,E)$, then the map $\chi$ is an isometry 
by choosing $\Gamma_{E}\coloneqq \{\|\cdot\|_{E}\}$. 
\end{rem}

Next, we prove the surjectivity of the map $\chi$ if $E$ is complete. The proof is quite similar to the one given in 
\cite[Theorem 14 (i), p.~1524]{kruse2017}. We identify $E$ with a linear subspace of the algebraic dual $E'^{\star}$ of $E'$
by the canonical linear injection $x\longmapsto [e'\mapsto e'(x)]\eqqcolon\langle x, e'\rangle$.

\begin{thm}\label{thm:linearisation_full}
Let $\F$ be a locally convex Hausdorff space of $\K$-valued functions on a non-empty set $\Omega$, 
$(\delta,Y,T)$ a strong linearisation of $\F$ such that $Y$ is quasi-barrelled and $E$ a complete locally convex Hausdorff space 
over the field $\K$. Then the map 
\[
\chi\colon L_{b}(Y,E)\to\FE_{\sigma,b},\;\chi(u)\coloneqq u\circ \delta,
\]
is a topological isomorphism and its inverse fulfils
\[
\langle \chi^{-1}(f)(y),e'\rangle=T(e'\circ f)(y),\quad f\in \FE_{\sigma,b},\,y\in Y,\,e'\in E'.
\]
In particular, $\FE_{\sigma,b}$ is quasi-complete. If $Y$ is even bornological, then the space $\FE_{\sigma,b}$ is complete.
\end{thm}
\begin{proof}
Due to \prettyref{prop:linearisation_into} (a) we only need to show that $\chi$ is surjective and its inverse continuous. 
Fix $f\in \FE_{\sigma,b}$. For all $y\in Y$ the map $\Psi(f)(y)\colon E'\to \K$, 
$\langle\Psi(f)(y),e'\rangle\coloneqq T(e'\circ f)(y)$, is clearly linear, thus $\Psi(f)(y)\in E'^{\star}$. 
Let $\Gamma_{E}$ denote a fundamental system of seminorms of $E$. We set for $p\in\Gamma_{E}$
\[
|z|_{U_{p}^{\circ}}\coloneqq \sup_{e'\in U_{p}^{\circ}}|z(e')|\leq\infty,\quad z\in {E'}^{\star},
\]
and note that $p(x)=|\langle\cdot,x\rangle|_{U_{p}^{\circ}}$ for every $x\in E$. Further, we define 
$R_{f}(e')\coloneqq e'\circ f$ for $e'\in E'$ and denote by $\Gamma_{Y}$ a fundamental system of seminorms of $Y$. 
We observe that $R_{f}(U_{p}^{\circ})$ is a bounded set in $\F$ and that 
there are $\widetilde{C}\geq 0$ and $\widetilde{q}\in\Gamma_{Y}$ such that 
\begin{equation}\label{eq:linearisation}
  |\Psi(f)(y)|_{U_{p}^{\circ}}
= \sup_{h\in R_{f}(U_{p}^{\circ})}|T(h)(y)|
= \sup_{h\in R_{f}(U_{p}^{\circ})}|\mathcal{J}_{Y}(y)(T(h))|
\leq \widetilde{C}\widetilde{q}(y)
\end{equation}
for all $y\in Y$ since $Y$ is a quasi-barrelled space and thus the canonical evaluation map 
$\mathcal{J}_{Y}\colon Y\to (Y_{b}')_{b}'$ continuous by \cite[11.2.2. Proposition, p.~222]{jarchow1981}.
Next, we show that $\Psi(f)(y)\in E$ for all $y\in Y$. First, we remark that 
\[
\langle\Psi(f)(\delta(x)),e'\rangle=T(e'\circ f)(\delta(x))=(e'\circ f)(x)=\langle f(x),e'\rangle,\quad x\in\Omega,
\]
for every $e'\in E'$, yielding $\Psi(f)(\delta(x))\in E$ for every $x\in\Omega$. 
By \cite[Proposition 2.7, p.~1596]{kruse2023a} the span of $\{\delta(x)\;|\;x\in\Omega\}$ is dense 
in $Y$. Thus for every $y\in Y$ there is a net $(y_{\iota})_{\iota\in I}$ in this span such that it converges to $y$ and 
$\Psi(f)(y_{\iota})\in E$ for each $\iota\in I$. For every $p\in\Gamma_{E}$ we have
\[
     |\Psi(f)(y_{\iota})-\Psi(f)(y)|_{U_{p}^{\circ}}
\underset{\eqref{eq:linearisation}}{\leq} \widetilde{C}\widetilde{q}(y_{\iota}-y)
\to 0.
\]
Hence $(\Psi(f)(y_{\iota}))_{\iota\in I}$ is a Cauchy net in the complete space $E$ with a limit 
$g\in E$. For every $p\in\Gamma_{E}$ we get
\begin{align*}
 |g-\Psi(f)(y)|_{U_{p}^{\circ}}
 &\leq  |g-\Psi(f)(y_{\iota})|_{U_{p}^{\circ}} + |\Psi(f)(y_{\iota})-\Psi(f)(y)|_{U_{p}^{\circ}}\\
 &\underset{\mathclap{\eqref{eq:linearisation}}}{\leq}  |g-\Psi(f)(y_{\iota})|_{U_{p}^{\circ}} 
  + \widetilde{C}\widetilde{q}(y_{\iota}-y)
\end{align*}
and deduce that $\Psi(f)(y)=g\in E$. In combination with \eqref{eq:linearisation} we derive that 
$\Psi(f)\in L(Y,E)$. Finally, we note that 
\[
\chi(\Psi(f))(x)=\Psi(f)(\delta(x))=f(x),\quad x\in\Omega,
\]
implying the surjectivity of $\chi$. 

Let $\Gamma_{\F}$ denote a fundamental system of seminorms of $\F$. 
Let $B\subset Y$ be bounded and $p\in\Gamma_{E}$. Due to the continuity of $T\colon\F\to Y_{b}'$ there are $C\geq 0$ and 
$q\in\Gamma_{\F}$ such that 
\begin{align*}
 \sup_{y\in B}p(\Psi(f)(y))
&=\sup_{y\in B}\sup_{e'\in U_{p}^{\circ}}|e'(\Psi(f)(y))|
 =\sup_{e'\in U_{p}^{\circ}}\sup_{y\in B}|T(e'\circ f)(y)|
 \leq C \sup_{e'\in U_{p}^{\circ}}q(e'\circ f)\\
&=C\|f\|_{\sigma,q,p},
\end{align*}
where we used \cite[Proposition 22.14, p.~256]{meisevogt1997} for the first equation, which implies the continuity of $\Psi=\chi^{-1}$. 

The space $L_{b}(Y,E)$ is quasi-complete, thus $\FE_{\sigma,b}$ as well, by \cite[\S 39, 6.(5), p.~144]{koethe1979} since $Y$ is 
quasi-barrelled and $E$ complete. If $Y$ is even bornological, then $L_{b}(Y,E)$ is complete by \cite[\S 39, 6.(4), p.~143]{koethe1979} and hence $\FE_{\sigma,b}$ as well.
\end{proof}

Due to weak-strong principles for holomorphic, harmonic and Lipschitz continuous functions special cases 
of \prettyref{thm:linearisation_full} and \prettyref{rem:linearisation_isometry_Banach_valued} for Banach spaces $E$ over $\C$ 
were already obtained before, for instance in 
\cite[Theorem 2.1, p.~869]{mujica1991} for the Banach space $\F=\mathcal{H}^{\infty}(\Omega)$ of bounded holomorphic functions on 
an open subset $\Omega$ of a Banach space, in \cite[Lemma 10, p.~243]{bonet2001} for a Banach space $\mathcal{F}(\D)$ 
of holomorphic functions on the open unit disc $\D\subset\C$ such that its closed unit ball is $\tau_{\operatorname{co}}$-compact, 
in \cite[Lemma 5.2, p.~14]{laitila2006} for a Banach space $\mathcal{F}(\D)$ of harmonic functions on $\D$ such that its 
closed unit ball is $\tau_{\operatorname{co}}$-compact, in \cite[Proposition 6, p.~3]{jorda2013}, 
which also covers \cite[Theorem 3.1 (Linearization Theorem), p.~128]{gupta2016} 
and \cite[Proposition 2.3 (c), p.~3029]{aron2024}, for a closed subspace $\F$ of the Banach space $\mathcal{H}v(\Omega)$ such that 
its closed unit ball is $\tau_{\operatorname{co}}$-compact where 
$\Omega$ is an open connected subset of a Banach space and $v\colon\Omega\to (0,\infty)$ a continuous function, 
and in \cite[Theorem 1, p.~239--240]{galindo1992} where $\F=\mathcal{H}_{b}(\Omega)$ is the space of holomorphic functions
defined on a balanced open subset $\Omega$ of a complex normed space which are bounded on $\Omega$-bounded sets. 
$\mathcal{H}_{b}(\Omega)$ is a Fr\'echet space if equipped with the topology of uniform convergence on $\Omega$-bounded sets, and 
we note that its predual $Y\coloneqq\mathfrak{P}_{b}(\Omega)$ in \cite[p.~238--239]{galindo1992} 
is an inductive limit of a sequence of Banach spaces and thus ultrabornological, in particular barrelled, by construction. 
For Banach spaces $E$ over $\C$ \prettyref{thm:linearisation_full} and \prettyref{rem:linearisation_isometry_Banach_valued} 
also generalise \cite[Lemma 2.3.1, p.~67]{beltran2014} where $(\delta,Y,T)\coloneqq (\Delta,G,J)$ is a strong Banach linearisation of a 
Banach space $\F$ of $\C$-valued functions on a non-empty set $\Omega$ with Banach predual $(G,J)$ such that 
$\delta_{x}\in G$ for all $x\in\Omega$, as well as \cite[Corollary 2.3.4, p.~68]{beltran2014} where 
$\F=\mathcal{HV}(\Omega)$ is the weighted Fr\'echet space of holomorphic functions on a complex Banach space $\Omega$ 
and $\mathcal{V}\coloneqq(v_{n})_{n\in\N}$ an increasing sequence of continuous functions $v_{n}\colon\Omega\to(0,\infty)$. 
Further, \prettyref{thm:linearisation_full} for Banach spaces $E$ over $\C$ also covers \cite[Theorem 2, p.~283]{beltran2012} 
where $\F=\mathcal{VH}(\Omega)$ is the inductive limit of the weighted Banach spaces $\mathcal{H}v_{n}(\Omega)$ 
of holomorphic functions on a complex Banach space $\Omega$ and $\mathcal{V}\coloneqq(v_{n})_{n\in\N}$ 
a decreasing sequence of continuous functions $v_{n}\colon\Omega\to(0,\infty)$ 
if $\mathcal{VH}(\Omega)$ satisfies $(\operatorname{CNC)}$ for $\tau_{\operatorname{co}}$ or 
$(\mathcal{VH}(\Omega)_{\mathcal{B},\tau_{\operatorname{co}}},\beta)$ is distinguished by \cite[Corollary 4.7, p.~1610--1611]{kruse2023a}. 
For complete locally convex Hausdorff spaces $E$ over $\C$ \prettyref{thm:linearisation_full} also generalises
\cite[3.7 Proposition, p.~292]{bierstedt1993} combined with \cite[1.5 Theorem (e), p.~277--278]{bierstedt1993} 
where $\F=\mathcal{HV}(\Omega)$ is the weighted space of holomorphic functions on 
a balanced open subset of $\Omega\subset\C^{d}$ and $\mathcal{V}$ a family of radial upper semicontinuous functions 
$v\colon\Omega\to [0,\infty)$ such that the weighted topology $\tau_{\mathcal{V}}$ induced by $\mathcal{V}$ fulfils 
$\tau_{\operatorname{co}}\leq\tau_{\mathcal{V}}$ (see \cite[p.~272, 274]{bierstedt1993}), $\mathcal{HV}(\Omega)$ is bornological 
and the polynomials are contained in 
$\mathcal{HV}_{0}(\Omega)\coloneqq\{f\in\mathcal{HV}(\Omega)\;|\;\forall\;v\in \mathcal{V}:\;fv\;\text{vanishes at}\;\infty\}$. 
However, we note that in contrast to \prettyref{thm:linearisation_full} quasi-complete $E$ are allowed in 
\cite[3.7 Proposition, p.~292]{bierstedt1993}. 
\prettyref{thm:linearisation_full} for complete locally convex Hausdorff spaces $E$ over $\C$ also improves 
\cite[Theorem 3.3, p.~35]{bonet2002} where $\F=\mathcal{HV}(\Omega)$ 
is bornological, $\Omega$ an open connected subset of $\C^{d}$ and $\mathcal{V}$ a point-detecting Nachbin family of 
continuous non-negative functions on $\Omega$. 
If $E$ is a complete locally convex Hausdorff space and $\F=\mathscr{L}(F\times G,\K)$ the space of continuous bilinear 
forms on the product $\Omega\coloneqq F\times G$ of two locally convex Hausdorff spaces $F$ and $G$, then 
\prettyref{thm:linearisation_full} also covers \cite[Chap.~I, \S1, n$^\circ$1, Proposition 2, p.~30--31]{grothendieck1966} 
(cf.~\cite[15.1.2 Theorem, p.~325]{jarchow1981}) if the projective tensor product $Y\coloneqq F \otimes_{\pi} G$ 
is quasi-barrelled (instead of $F\otimes_{\pi} G$ one may also use its completion $F\widehat{\otimes}_{\pi} G$ 
by \cite[3.4.2 Theorem, p.~61--62]{jarchow1981}). 
For instance, $F\otimes_{\pi} G$ is (quasi-)barrelled (and $F\widehat{\otimes}_{\pi} G$ barrelled by 
\cite[11.3.1 Proposition (e), p.~223]{jarchow1981}) if $F$ and $G$ are metrisable and barrelled by 
\cite[15.6.6 Proposition, p.~337]{jarchow1981}, or if $F$ and $G$ are quasi-barrelled DF-spaces by 
\cite[15.6.8 Proposition, p.~338]{jarchow1981}. 
However, we emphasize that \cite[Chap.~I, \S1, n$^\circ$1, Proposition 2, p.~30--31]{grothendieck1966} 
in contrast to \prettyref{thm:linearisation_full} does not need the restrictions that $E$ is complete and $F \otimes_{\pi} G$ 
quasi-barrelled. 
Moreover, \prettyref{thm:linearisation_full} also complements \cite[Theorem 3, p.~690]{carando2004} where a special 
continuous linearisation $(e,\mathcal{F}_{\ast}(\Omega),L)$ of a linear space $\F$ of scalar-valued continuous functions on a 
non-empty topological Hausdorff space $\Omega$ instead of a strong linearisation is considered 
and it is shown that the map $L(\mathcal{F}_{\ast}(\Omega),E)\to \FE_{\sigma}\cap\mathcal{C}(\Omega,E)$, $u\mapsto u\circ e$, 
is an algebraic isomorphism if $E$ is a complete locally convex Hausdorff space. 
Here, $\mathcal{C}(\Omega,E)$ denotes the space of continuous maps from $\Omega$ to $E$. We refer the reader to 
\cite[p.~182--184]{jaramillo2009} for a summary of the construction of $(e,\mathcal{F}_{\ast}(\Omega),L)$ .

\begin{rem}\label{rem:linearisation_predual}
Let $\F$ be a locally convex Hausdorff space of $\K$-valued functions on a non-empty set $\Omega$, 
$Y$ a quasi-barrelled locally convex Hausdorff space over the field $\K$ and $(\delta,Y,T)$ a strong linearisation of $\F$.
For any complete locally convex Hausdorff space $E$ over the field $\K$ we have $u\circ\delta\in\FE_{\sigma,b}$ for all 
$u\in L_{b}(Y,E)$ and there is a topological isomorphism $T_{E}\colon\FE_{\sigma,b}\to L_{b}(Y,E)$, 
namely $T_{E}\coloneqq\chi^{-1}=\Psi$, with $T_{E}(f)\circ \delta=f$ for all $f\in\FE_{\sigma,b}$ 
by \prettyref{thm:linearisation_full}. 
Looking at \prettyref{defn:linearisation}, we may consider $(\delta,L_{b}(Y,E),T_{E})$ as a strong linearisation of $\FE_{\sigma,b}$. 
In particular, the following diagram commutes:
\[
\xymatrix{
\Omega \ar[d]_{\delta} \ar[r]^{f} &  E  \\
Y \ar[ur]_{T_{E}(f)} &
}
\]
If $Y$ is complete, then $Y$ is also a \emph{free object generated by $\Omega$ 
in the category} $\mathcal{C}_{2}$ of complete locally convex Hausdorff spaces with continuous linear operators as morphisms 
in the sense of \cite[p.~100]{garcia2024} once the spaces $\Omega$ belong to a category $\mathcal{C}_{1}$ 
such that $\delta$ and any $f\in\FE_{\sigma,b}$ is a morphism of $\mathcal{C}_{1}$ 
and there is a ``forgetful'' functor from $\mathcal{C}_{2}$ to $\mathcal{C}_{1}$.
\end{rem}

Next, we generalise the extension result \cite[Theorem 10, p.~5]{jorda2013}. 
For this purpose we need to recall some definitions.
Let $E$ be a locally convex Hausdorff space. A linear subspace $G\subset E'$ is said to \emph{determine boundedness} of $E$
if every $\sigma(E,G)$-bounded set $B\subset E$ is already bounded in $E$ (see \cite[p.~231]{bonet2007}). In particular, such 
a $G$ is $\sigma(E',E)$-dense in $E'$. Further, by Mackey's theorem $G\coloneqq E'$ determines boundedness. 
Another example is the following one. 

\begin{rem}
Let $(E,\tau)$ be a bornological locally convex Hausdorff space, $\mathcal{B}$ the family of $\tau$-bounded sets, 
$\widetilde{\tau}$ a locally convex Hausdorff topology on $E$ and $\widetilde{\gamma}$ denote the finest locally convex Hausdorff topology on $E$ which coincides with $\widetilde{\tau}$ on $\tau$-bounded sets. Suppose that a subset of $E$ is $\tau$-bounded if and only if it is $\widetilde{\gamma}$-bounded (see \cite[Definition 3.12, p.~1600]{kruse2023a}). 
Then $(E,\widetilde{\gamma})'$ determines boundedness of $(E,\tau)$. 
Indeed, $(E,\widetilde{\gamma})'\subset (E,\tau)'$ by \cite[Remark 3.13 (c), p.~1600--1601]{kruse2023a}. 
Let $B\subset E$ be $\sigma(E,(E,\widetilde{\gamma})')$-bounded. Then $B$ is $\widetilde{\gamma}$-bounded by 
Mackey's theorem. It follows that $B$ is $\tau$-bounded by assumption.
\end{rem}

Moreover, let $(\F,\tau)$ be a locally convex Hausdorff space of $\K$-valued functions on a non-empty set $\Omega$. 
A set $U\subset\Omega$ is called a \emph{set of uniqueness} 
for $\F$ if for each $f\in\F$ the validity of $f(x)=0$ for all $x\in U$ implies $f=0$ on $\Omega$ 
(see \cite[p.~3]{jorda2013}). If $(\F,\tau)$ is bornological and satisfies $\operatorname{(BBC)}$ and $\operatorname{(CNC)}$ for 
some $\widetilde{\tau}$ such that $\tau_{\operatorname{p}}\leq\widetilde{\tau}$, 
then $U\subset\Omega$ is a set of uniqueness for $\F$ if and only if the span of 
$\{\delta_{x}\;|\;x\in U\}$ is $\sigma(\F_{\mathcal{B},\widetilde{\tau}}',\F)$-dense. 
For instance, a sequence $U\coloneqq(z_{n})_{n\in\N}\subset\D$ of distinct elements is a set of uniqueness for 
$\mathcal{F}(\D)=\mathcal{H}^{\infty}(\D)$ if and only if it satisfies the Blaschke condition $\sum_{n\in\N}(1-|z_{n}|)=\infty$ 
(see e.g.~\cite[15.23 Theorem, p.~303]{rudin1970}). Further examples of sets of uniqueness for the spaces 
$\mathcal{C}_{P}v(\Omega)$ are given in \cite[p.~102--103]{kruse2023}. 
Now, we only need to adapt the proof of \cite[Theorem 10, p.~5]{jorda2013} a bit to get the following theorem. 

\begin{thm}\label{thm:extension}
Let $(\F,\tau)$ be a complete bornological DF-space of $\K$-valued functions on a non-empty set $\Omega$ 
satisfying $\operatorname{(BBC)}$ and $\operatorname{(CNC)}$ for some $\widetilde{\tau}$ such that 
$\tau_{\operatorname{p}}\leq\widetilde{\tau}$, $U\subset\Omega$ a set of uniqueness for $\F$, 
$E$ a complete locally convex Hausdorff space 
and $G\subset E'$ determine boundedness. If $f\colon U\to E$ is a function such that $e'\circ f$ admits an extension 
$f_{e'}\in\F$ for each $e'\in G$, then there exists a unique extension $F\in\FE_{\sigma}$ of $f$.
\end{thm}
\begin{proof}
Let $Y_{U}$ denote the span of $\{\delta_{x}\;|\;x\in U\}$. Since $U$ is a set of uniqueness 
for $\F$ and $\tau_{\operatorname{p}}\leq\widetilde{\tau}$, $Y_{U}$ is $\sigma(\F_{\mathcal{B},\widetilde{\tau}}',\F)$-dense 
and thus also $\beta$-dense by \cite[8.2.5 Proposition, p.~149]{jarchow1981} as
$(\F_{\mathcal{B},\widetilde{\tau}}',\sigma(\F_{\mathcal{B},\widetilde{\tau}}',\F))'=\F
=(\F_{\mathcal{B},\widetilde{\tau}}',\beta)'$ (as linear spaces). 
We note that the linear map $A\colon Y_{U}\to E$ determined by $A(\delta_{x})\coloneqq f(x)$ for $x\in U$
is well-defined because $G$ is $\sigma(E',E)$-dense. Let $B\subset\F_{\mathcal{B},\widetilde{\tau}}'$ be $\beta$-bounded, 
$x'\in B\cap Y_{U}$ and $e'\in G$. Then $x'$ can be represented as $x'=\sum_{x\in U}a_{x}\delta_{x}$ with finitely many 
non-zero $a_{x}\in\K$ and
\[
 |e'(A(x'))|
=\Bigl|\sum_{x\in U}a_{x}e'(f(x))\Bigr|
=\Bigl|\sum_{x\in U}a_{x}f_{e'}(x)\Bigr|
=\Bigl|\Bigl(\sum_{x\in U}a_{x}\delta_{x}\Bigr)(f_{e'})\Bigr|
=|x'(f_{e'})|.
\]
Hence, by the $\beta$-boundedness of $B$ there is some $C\geq 0$ such that $|e'(A(x'))|\leq C$ for all $x'\in B\cap Y_{U}$. We deduce that the set 
$A(B\cap Y_{U})$ is $\sigma(E,G)$-bounded and thus bounded in $E$ because 
$G$ determines boundedness. We observe that $(\F_{\mathcal{B},\widetilde{\tau}}',\beta)$ is a Fr\'echet space 
by \prettyref{prop:predual_complete} (b) and so its linear subspace $Y_{U}$ is metrisable. 
We conclude that $A\colon (Y_{U},\beta_{\mid Y_{U}})\to E$ 
is continuous by \cite[Proposition 24.10, p.~282]{meisevogt1997} as the metrisable space 
$(Y_{U},\beta_{\mid Y_{U}})$ is bornological by 
\cite[Proposition 24.13, p.~283]{meisevogt1997}. Since $Y_{U}$ is 
$\beta$-dense in $\F_{\mathcal{B},\widetilde{\tau}}'$, there is a unique continuous linear 
extension $\widetilde{A}\colon \F_{\mathcal{B},\widetilde{\tau}}'\to E$ of $A$ by \cite[3.4.2 Theorem, p.~61--62]{jarchow1981}. 
Setting $F\coloneqq\chi\circ\widetilde{A}\in\FE_{\sigma}$ by \prettyref{cor:scb_linearisation} (a) 
and \prettyref{thm:linearisation_full}, we observe that 
\[
F(x)=(\chi\circ\widetilde{A})(x)=\widetilde{A}(\delta_{x})=A(\delta_{x})=f(x)
\]
for all $x\in U$, which proves the existence of 
the extension of $f$. The uniqueness of the extension follows from $U$ being a set 
of uniqueness for $\F$ and $G$ being $\sigma(E',E)$-dense.
\end{proof}

In particular, \prettyref{thm:extension} is applicable to the spaces $\F=\mathcal{VC}_{P}(\Omega)$ from 
\prettyref{ex:cont_strong_lin} (ii) by \cite[Corollary 4.7, p.~1610--1611]{kruse2023a}. 

\section{Equivalence and uniqueness of preduals}
\label{sect:equivalence_predual}

Having two strong linearisations $(\delta,Y,T)$ and $(\widetilde{\delta},Z,\widetilde{T})$ of a locally convex Hausdorff space $\F$ 
of $\K$-valued functions on a non-empty set $\Omega$, one might suspect that the preduals $(Y,T)$ and $(Z,\widetilde{T})$ 
are equivalent. This section is dedicated to deriving necessary and sufficient conditions for this to happen. Our candidate 
for the topological isomorphism $\lambda\colon Y\to Z$ such that $\lambda^{t}=T\circ \widetilde{T}^{-1}$ is the map 
$T_{Z}(\widetilde{\delta})$ with $T_{Z}$ from \prettyref{rem:linearisation_predual} for complete $Z$ if 
$\widetilde{\delta}\in\mathcal{F}(\Omega,Z)_{\sigma,b}$. 

\begin{prop}\label{prop:linearisation_scalar_valued}
Let $\F$ be a locally convex Hausdorff space of $\K$-valued functions on a non-empty set $\Omega$ and 
$(\delta,Y,T)$ a strong linearisation of $\F$ such that $Y$ is quasi-barrelled.
Let $(\widetilde{\delta},Z,\widetilde{T})$ be a linearisation of $\F$ such that $Z$ is complete and 
$\widetilde{\delta}\in\mathcal{F}(\Omega,Z)_{\sigma,b}$. 
Then the following assertions hold.
\begin{enumerate} 
\item[(a)] The map $T_{Z}(\widetilde{\delta})\colon Y\to Z$ is linear, continuous, injective and has dense range, 
$\widetilde{T}^{-1}\colon Z_{b}'\to \F$ is continuous and $T_{Z}(\widetilde{\delta})^{t}=T\circ \widetilde{T}^{-1}$.
\item[(b)] If 
\begin{enumerate}
\item[(i)] $Z_{b}'$ is webbed and $\F$ ultrabornological, or 
\item[(ii)] $Z_{b}'$ is $B_{r}$-complete and $\F$ barrelled, 
\end{enumerate}
then $\widetilde{T}\colon \F\to Z_{b}'$ is a topological isomorphism. 
\item[(c)] If $Y$ is complete, $Z$ barrelled and 
$\widetilde{T}\colon \F\to Z_{b}'$ continuous, then $T_{Z}(\widetilde{\delta})$ is surjective. 
\item[(d)] If $\widetilde{T}\colon \F\to Z_{b}'$ is continuous and 
\begin{enumerate}
\item[(i)] $Y$ is complete and webbed and $Z$ ultrabornological, or 
\item[(ii)] $Y$ is $B_{r}$-complete and $Z$ barrelled,
\end{enumerate}
then $T_{Z}(\widetilde{\delta})$ is a topological isomorphism.
\end{enumerate}
\end{prop}
\begin{proof}
(a) We note that $T_{Z}(\widetilde{\delta})$ is a continuous linear map by \prettyref{thm:linearisation_full} and thus 
$T_{Z}(\widetilde{\delta})^{t}\colon Z_{b}'\to Y_{b}'$ as well. We observe that
\begin{align*}
  T_{Z}(\widetilde{\delta})^{t}(z')(y)
&=\langle z',T_{Z}(\widetilde{\delta})(y)\rangle
 =\langle \chi^{-1}(\widetilde{\delta})(y),z'\rangle
 = T(z'\circ\widetilde{\delta})(y)\\
&= T(\widetilde{T}(\widetilde{T}^{-1}(z'))\circ \widetilde{\delta})(y)
 =T(\widetilde{T}^{-1}(z'))(y)
\end{align*}
for all $z'\in Z'$ and $y\in Y$ by \prettyref{thm:linearisation_full} and using that 
$T_{Z}(\widetilde{\delta})=\chi^{-1}(\widetilde{\delta})$. In particular, $T_{Z}(\widetilde{\delta})^{t}\colon Z_{b}'\to Y_{b}'$ 
is a continuous algebraic isomorphism. 
Since $\widetilde{T}^{-1}=T^{-1}\circ T_{Z}(\widetilde{\delta})^{t}$, we get that 
$\widetilde{T}^{-1}\colon Z_{b}'\to \F$ is continuous. 

Further, let $y\in Y$ such that $T_{Z}(\widetilde{\delta})(y)=0$. For every $y'\in Y'$ there is $z'\in Z'$ with 
$T_{Z}(\widetilde{\delta})^{t}(z')=y'$ by the surjectivity of $T_{Z}(\widetilde{\delta})^{t}$, which implies 
\[
y'(y)=T_{Z}(\widetilde{\delta})^{t}(z')(y)=z'(T_{Z}(\widetilde{\delta})(y))=z'(0)=0
\]
for all $y'\in Y'$. Hence $y=0$ by the Hahn--Banach theorem and $T_{Z}(\widetilde{\delta})$ is injective. 

Moreover, (the restriction of) $T_{Z}(\widetilde{\delta})$ is a bijective map from the span of 
$\{\delta(x)\;|\;x\in\Omega\}$ to the span of $\{\widetilde{\delta}(x)\;|\;x\in\Omega\}$, 
which is dense in $Z$ by \cite[Proposition 2.7, p.~1596]{kruse2023a}, because 
$T_{Z}(\widetilde{\delta})(\delta(x))=\widetilde{\delta}(x)$ for all $x\in\Omega$. Thus $T_{Z}(\widetilde{\delta})$ has dense range.

(b) This statement follows from part (a) and \cite[Open mapping theorem 24.30, p.~289]{meisevogt1997} in case (i) and 
\cite[11.1.7 Theorem (b), p.~221]{jarchow1981} in case (ii). 

(c) We denote by $j_{Y}\colon Y\to(Y_{b}')'$ and 
$j_{Z}\colon Z\to(Z_{b}')'$ the canonical linear injections and observe that the map 
$T_{Z}(\widetilde{\delta})^{tt}\colon (Y_{b}')_{b}'\to (Z_{b}')_{b}' $ is linear, continuous and 
bijective and its inverse fulfils 
\[
 (T_{Z}(\widetilde{\delta})^{tt})^{-1}
=((T_{Z}(\widetilde{\delta})^{t})^{-1})^{t}
=(\widetilde{T}\circ T^{-1})^{t}
\]
by part (a). We note that $j_{Y}\colon Y\to(Y_{b}')_{b}'$ and 
$j_{Z}\colon Z\to(Z_{b}')_{b}'$ are topological isomorphisms into 
by \cite[11.2.2 Proposition, p.~222]{jarchow1981} as $Y$ and $Z$ are (quasi-)barrelled. 
Since $\widetilde{T}\colon \F\to Z_{b}'$ is continuous by assumption, we get that 
$(T_{Z}(\widetilde{\delta})^{tt})^{-1}$ is also continuous and hence $T_{Z}(\widetilde{\delta})^{tt}$ a topological isomorphism. 

Let $z\in Z$. Since the span of $\{\widetilde{\delta}(x)\;|\;x\in\Omega\}$ is dense in $Z$ 
by \cite[Proposition 2.7, p.~1596]{kruse2023a}, there is a net $(z_{\iota})_{\iota\in I}$ 
converging to $z$ and where all $z_{\iota}$ can be represented as $z_{\iota}=\sum_{x\in\Omega}a_{x,\iota}\widetilde{\delta}(x)$ with finitely many 
non-zero $a_{x,\iota}\in\K$. Using that $T_{Z}(\widetilde{\delta})''\circ j_{Y}=j_{Z}\circ T_{Z}(\widetilde{\delta})$ 
and setting $y_{\iota}\coloneqq \sum_{x\in\Omega}a_{x,\iota}\delta(x)\in Y$ for $\iota\in I$, we get 
\[
 T_{Z}(\widetilde{\delta})^{tt}(j_{Y}(y_{\iota}))
=(j_{Z}\circ T_{Z}(\widetilde{\delta}))(y_{\iota})
=j_{Z}(\sum_{x\in\Omega}a_{x,\iota}\widetilde{\delta}(x))
=j_{Z}(z_{\iota})
\]
and so
\[
j_{Y}(y_{\iota})=(T_{Z}(\widetilde{\delta})^{tt})^{-1}(j_{Z}(z_{\iota}))
\]
for all $\iota\in I$, which implies that the net $(j_{Y}(y_{\iota}))_{\iota\in I}$ converges to $(T_{Z}(\widetilde{\delta})^{tt})^{-1}(j_{Z}(z))$ 
in $(Y_{b}')_{b}'$ since $T_{Z}(\widetilde{\delta})^{tt}$ is a topological isomorphism and $(j_{Z}(z_{\iota}))_{\iota\in I}$ converges to 
$j_{Z}(z)$ due to the barrelledness of $Z$. The quasi-barrelledness of $Y$ implies that $(y_{\iota})_{\iota\in I}$ is a Cauchy net 
in $Y$. From the completeness of $Y$ we deduce that $(y_{\iota})_{\iota\in I}$ converges to some $y\in Y$, yielding 
\[
 (j_{Z}\circ T_{Z}(\widetilde{\delta}))(y)
=(T_{Z}(\widetilde{\delta})^{tt}\circ j_{Y})(y)
=j_{Z}(z)
\] 
and thus $T_{Z}(\widetilde{\delta})(y)=z$ by the injectivity of $j_{Z}$, which means that $T_{Z}(\widetilde{\delta})$ is surjective. 

(d) The statement follows from part (a) and (c) and \cite[Open mapping theorem 24.30, p.~289]{meisevogt1997} 
in case (i) and \cite[11.1.7 Theorem (b), p.~221]{jarchow1981} in case (ii) 
combined with the observations that ultrabornological spaces are barrelled, 
and $B_{r}$-complete spaces are complete by \cite[9.5.1 Proposition (b), p.~183]{jarchow1981}. 
\end{proof}

\prettyref{prop:linearisation_scalar_valued} complements \cite[Corollary 2, p.~695]{carando2004} where for a special 
continuous linearisation $(e,\mathcal{F}_{\ast}(\Omega),L)$ of $\F$ instead of a strong linearisation it is shown that 
$\mathcal{F}_{\ast}(\Omega)$ is topologically isomorphic to $Z$ for any other continuous linearisation 
$(\widetilde{e},Z,\widetilde{T})$ of $\F$ such that $Z$ is a Fr\'echet space. 
Looking at part (b), we note that $Z_{b}'$ is a Fr\'echet space, so $B_{r}$-complete and webbed, by 
\cite[9.5.2 Krein-\u{S}mulian Theorem, p.~184]{jarchow1981} and \cite[12.4.2 Theorem, p.~258]{jarchow1981} 
if $Z$ is a complete gDF-space. Further, $Z_{b}'$ is a complete DF-space, so webbed, by \cite[12.4.5 Theorem, p.~260]{jarchow1981} 
and \cite[12.4.6 Proposition, p.~260]{jarchow1981} if $Z$ is a Fr\'echet space. 

\begin{prop}\label{prop:linearisation_isometry}
Let $(\F,\|\cdot\|)$ be a Banach space of $\K$-valued functions on a non-empty set $\Omega$ such that 
$\delta_{x}\in (\F,\|\cdot\|)'$ for all $x\in\Omega$, $(\delta,Y,T)$ a linearisation of 
$\F$ such that $(Y,\|\cdot\|_{Y})$ is a Banach space and $T\colon (\F,\|\cdot\|)\to ((Y,\|\cdot\|_{Y})',\|\cdot\|_{Y'})$ an isometry. 
If $(\widetilde{\delta},Z,\widetilde{T})$ is a linearisation of $\F$ such that $(Z,\|\cdot\|_{Z})$ is a Banach space and 
the map $\widetilde{T}\colon(\F,\|\cdot\|)\to ((Z,\|\cdot\|_{Z})',\|\cdot\|_{Z'})$ an isometry, 
then the map $T_{Z}(\widetilde{\delta})\colon (Y,\|\cdot\|_{Y})\to (Z,\|\cdot\|_{Z})$ is an isometric isomorphism.
\end{prop}
\begin{proof}
First, we note that $\mathcal{F}(\Omega,Z)_{\sigma,b}=\mathcal{F}(\Omega,Z)_{\sigma}$ by \prettyref{rem:BC_space} because 
the Banach space $(\F,\|\cdot\|)$ is a BC-space. Due to \prettyref{prop:linearisation_scalar_valued} (a) and (d) we know that 
$T_{Z}(\widetilde{\delta})$ is a topological isomorphism and $T_{Z}(\widetilde{\delta})^{t}=T\circ \widetilde{T}^{-1}$. 
Thus $T_{Z}(\widetilde{\delta})^{t}\colon ((Z,\|\cdot\|_{Z})',\|\cdot\|_{Z'})\to ((Y,\|\cdot\|_{Y})',\|\cdot\|_{Y'})$ is an isometric 
isomorphism since $T$ and $\widetilde{T}$ are isometries. 
Hence we have $T_{Z}(\widetilde{\delta})^{t}(B_{\|\cdot\|_{Z'}})=B_{\|\cdot\|_{Y'}}$ and 
\begin{align*}
  \|T_{Z}(\widetilde{\delta})(y)\|_{Z}
&=\sup_{z'\in B_{\|\cdot\|_{Z'}}}|z'(T_{Z}(\widetilde{\delta})(y))|
 =\sup_{z'\in B_{\|\cdot\|_{Z'}}}|T_{Z}(\widetilde{\delta})^{t}(z')(y)|
 =\sup_{y'\in B_{\|\cdot\|_{Y'}}}|y'(y)|\\
&=\|y\|_{Y}
\end{align*}
for all $y\in Y$. Thus $T_{Z}(\widetilde{\delta})$ is an isometry. 
\end{proof}

\prettyref{prop:linearisation_scalar_valued} (d) combined with \prettyref{prop:linearisation_isometry} improves 
\cite[Proposition 2.3 (c), p.~3029]{aron2024}, \cite[Theorem 3.1 (Linearization Theorem), p.~128]{gupta2016}, 
\cite[Theorem 2.1, p.~869]{mujica1991} and \cite[Theorem 3.5, p.~19]{quang2023} since it shows that the Banach space $Y$ of a 
strong linearisation $(\delta,Y,T)$ of a Banach space $\F$ such that $T$ is an isometry is uniquely determined up to an 
isometric isomorphism without the need of existence of isometric isomorphisms $\widetilde{T}_{E}\colon\FE_{\sigma}\to L(Z,E)$ 
such that $\widetilde{T}_{E}(f)\circ\widetilde{\delta}=f$ for all $f\in\F$ for every Banach space $E$ over $\K$. 
Further, \prettyref{prop:linearisation_scalar_valued} (d) combined with \prettyref{prop:linearisation_isometry} also implies 
the corresponding result for the (completion of the) projective tensor product of two Banach spaces given in 
\cite[Proposition 1.5 (Uniqueness of the Tensor Product), p.~6]{ryan2002} and \cite[Theorem 2.9, p.~22]{ryan2002}. 

\begin{prop}\label{prop:s_linearisation_in_equivalence_class}
Let $\F$ be a locally convex Hausdorff space of $\K$-valued functions on a non-empty set $\Omega$, 
$(\delta,Y,T)$ a strong linearisation of $\F$ and $Z$ a locally 
convex Hausdorff space. If there exists a topological isomorphism 
$\varphi\colon\F\to Z_{b}'$ such that $(Z,\varphi)\sim (Y,T)$, then there exists 
$\widetilde{\delta}\colon \Omega\to Z$ such that $(\widetilde{\delta},Z,\varphi)$ is a strong linearisation of $\F$.
\end{prop}
\begin{proof}
Since $(Z,\varphi)\sim (Y,T)$, there exists a topological isomorphism 
$\lambda\colon Y\to Z$ such that $\lambda^{t}= T\circ\varphi^{-1}$. 
We set $\widetilde{\delta}\colon\Omega\to Z$, $\widetilde{\delta}(x)\coloneqq\lambda(\delta(x))$, and note that 
\begin{align*}
 (\varphi(f)\circ\widetilde{\delta})(x)
&=\varphi(f)(\lambda(\delta(x)))
 =\lambda^{t}(\varphi(f))(\delta(x))
 =(T\circ\varphi^{-1})(\varphi(f))(\delta(x))\\
&=T(f)(\delta(x))
 =f(x)
\end{align*}
for all $f\in\F$ and $x\in\Omega$. Hence $(\widetilde{\delta},Z,\varphi)$ is a strong linearisation of $\F$.
\end{proof}

\begin{thm}\label{thm:s_linearisation_in_equivalence_class}
Let $\F$ be a locally convex Hausdorff space of $\K$-valued functions on a non-empty set $\Omega$, 
$(\delta,Y,T)$ a strong linearisation of $\F$, $(Z,\widetilde{T})$ a predual of $\F$ such that $Z$ is complete 
and let 
\begin{enumerate}
\item[(i)] $Y$ be complete, barrelled and webbed and $Z$ ultrabornological, or 
\item[(ii)] $Y$ be $B_{r}$-complete and barrelled and $Z$ barrelled. 
\end{enumerate}
Consider the following assertions.
\begin{enumerate}
\item[(a)] There exists $\widetilde{\delta}\in\mathcal{F}(\Omega,Z)_{\sigma,b}$ 
such that $(\widetilde{\delta},Z,\widetilde{T})$ is a strong linearisation of $\F$.
\item[(b)] It holds $(Z,\widetilde{T})\sim (Y,T)$.
\end{enumerate}
Then it holds (a)$\Rightarrow$(b). If $\F$ is a BC-space such that $\delta_{x}\in\F'$ for all $x\in\Omega$, then it holds (b)$\Rightarrow$(a).
\end{thm}
\begin{proof}
(a)$\Rightarrow$(b) Since $(\widetilde{\delta},Z,\widetilde{T})$ is a strong linearisation of $\F$, the topological isomorphism 
$\widetilde{T}\colon\F\to Z_{b}'$ fulfils $\widetilde{T}(f)\circ\widetilde{\delta}=f$ for all $f\in\F$. 
We conclude statement (b) from \prettyref{prop:linearisation_scalar_valued} (a) and (d). 

(b)$\Rightarrow$(a) This implication follows from \prettyref{prop:s_linearisation_in_equivalence_class} with 
$\varphi\coloneqq\widetilde{T}$ and \prettyref{rem:BC_space}.
\end{proof}

\begin{cor}
Let $(\F,\tau)$ be a bornological BC-space of $\K$-valued functions on a non-empty set $\Omega$ satisfying 
$\operatorname{(BBC)}$ and $\operatorname{(CNC)}$ for some $\tau_{\operatorname{p}}\leq\widetilde{\tau}$ and 
$(\F_{\mathcal{B},\widetilde{\tau}}',\beta)$ webbed, where $\mathcal{B}$ is the family of $\tau$-bounded sets, 
and $(Z,\widetilde{T})$ a predual of $(\F,\tau)$ such that $Z$ is complete and ultrabornological. 
Then the following assertions are equivalent.
\begin{enumerate}
\item[(a)] There exists $\widetilde{\delta}\colon \Omega\to Z$ such that $(\widetilde{\delta},Z,\widetilde{T})$ 
is a strong linearisation of $\F$.
\item[(b)] It holds $(Z,\widetilde{T})\sim(\F_{\mathcal{B},\widetilde{\tau}}',\mathcal{I})$.
\end{enumerate}
\end{cor}
\begin{proof}
Due to \prettyref{cor:scb_linearisation} (a) $(\Delta,\F_{\mathcal{B},\widetilde{\tau}}',\mathcal{I})$ is a strong 
complete barrelled linearisation of $\F$. Hence our statement follows from 
\prettyref{thm:s_linearisation_in_equivalence_class} (i) since $\Delta(x)=\delta_{x}\in(\F,\tau)'$ for all $x\in\Omega$ 
by \prettyref{prop:predual_complete} (a).
\end{proof}

In particular, $(\F_{\mathcal{B},\widetilde{\tau}}',\beta)$ is a complete DF-space (see \cite[Corollary 3.23, p.~1606]{kruse2023a}) 
and thus webbed by \cite[12.4.6 Proposition, p.~260]{jarchow1981} if $(\F,\tau)$ is a Fr\'echet space. 

\begin{cor}
Let $(\F,\tau)$ be a complete bornological DF-space of $\K$-valued functions on a non-empty set $\Omega$ satisfying 
$\operatorname{(BBC)}$ and $\operatorname{(CNC)}$ for some $\tau_{\operatorname{p}}\leq\widetilde{\tau}$ and $\mathcal{B}$ 
the family of $\tau$-bounded sets as well as $(Z,\widetilde{T})$ a predual of $(\F,\tau)$ such that $Z$ is complete and barrelled. 
Then the following assertions are equivalent.
\begin{enumerate}
\item[(a)] There exists $\widetilde{\delta}\colon\Omega\to Z$ such that $(\widetilde{\delta},Z,\widetilde{T})$ 
is a strong linearisation of $\F$.
\item[(b)] It holds $(Z,\widetilde{T})\sim(\F_{\mathcal{B},\widetilde{\tau}}',\mathcal{I})$.
\end{enumerate}
\end{cor}
\begin{proof}
Due to \prettyref{prop:predual_complete} (b) and \prettyref{cor:scb_linearisation} (a) 
$(\Delta,\F_{\mathcal{B},\widetilde{\tau}}',\mathcal{I})$ is a strong 
Fr\'echet linearisation of $\F$. By \prettyref{prop:predual_complete} (a) it holds $\Delta(x)=\delta_{x}\in(\F,\tau)'$ 
for all $x\in\Omega$. Since the Fr\'echet space $(\F_{\mathcal{B},\widetilde{\tau}}',\beta)$ is $B_{r}$-complete 
by \cite[9.5.2 Krein-\u{S}mulian Theorem, p.~184]{jarchow1981} and the complete DF-space $(\F,\tau)$ webbed by 
\cite[12.4.6 Proposition, p.~260]{jarchow1981}, so a BC-space, our statement follows 
from \prettyref{thm:s_linearisation_in_equivalence_class} (ii). 
\end{proof}

Using \cite[Propositions 3.16, 3.17, p.~1602]{kruse2023a}, \prettyref{prop:predual_complete} (c) and \prettyref{cor:scb_linearisation} (a), we get the following corollary for completely normable spaces $\F$.

\begin{cor}
Let $(\F,\tau)$ be a completely normable space of $\K$-valued functions on a non-empty set $\Omega$ satisfying 
$\operatorname{(BBC)}$ for some $\tau_{\operatorname{p}}\leq\widetilde{\tau}$ and $(Z,\widetilde{T})$ a predual of $(\F,\tau)$ 
such that $Z$ is complete and barrelled. Then the following assertions are equivalent.
\begin{enumerate}
\item[(a)] There exists $\widetilde{\delta}\colon\Omega\to Z$ such that $(\widetilde{\delta},Z,\widetilde{T})$ 
is a strong linearisation of $\F$.
\item[(b)] It holds $(Z,\widetilde{T})\sim((\F,\widetilde{\gamma})',\mathcal{I})$.
\end{enumerate}
\end{cor}

Let $(\delta,Y,T_{\K})$ be a strong linearisation of $\F$ such that $Y$ is quasi-barrelled. 
If we not only have a linearisation $(\widetilde{\delta},Z,\widetilde{T}_{\K})$ of $\F$ in the scalar-valued case 
as in the results above but also of $\mathcal{F}(\Omega,Y)_{\sigma}$ in the vector-valued case, we may get rid of some of the 
assumptions in \prettyref{prop:linearisation_scalar_valued} (d) on $\widetilde{T}_{\K}$, $Y$ and $Z$. 

\begin{prop}\label{prop:linearisation_predual_unique}
Let $\F$ be a locally convex Hausdorff space of $\K$-valued functions on a non-empty set $\Omega$, 
$(\delta,Y,T_{\K})$ a strong linearisation of $\F$ such that $Y$ is quasi-barrelled and $\delta\in\mathcal{F}(\Omega,Y)_{\sigma,b}$. 
If $Z$ is a complete locally convex Hausdorff space and $\widetilde{\delta}\in\mathcal{F}(\Omega,Z)_{\sigma,b}$ 
such that for all $E\in\{\K,Y\}$ there is an algebraic isomorphism 
$\widetilde{T}_{E}\colon\FE_{\sigma,b}\to L(Z,E)$ with $\widetilde{T}_{E}(f)\circ\widetilde{\delta}=f$ for all $f\in\FE_{\sigma,b}$, 
then the map $T_{Z}(\widetilde{\delta})\colon Y\to Z$ is a topological isomorphism with inverse 
$T_{Z}(\widetilde{\delta})^{-1}=\widetilde{T}_{Y}(\delta)$, the map $\widetilde{T}_{\K}$ is a topological isomorphism 
and $T_{Z}(\widetilde{\delta})^{t}=T_{\K}\circ \widetilde{T}_{\K}^{-1}$. 
In particular, it holds $(Z,\widetilde{T}_{\K})\sim (Y,T_{\K})$.
\end{prop}
\begin{proof}
First, we prove that $T_{Z}(\widetilde{\delta})\colon Y\to Z$ is a topological isomorphism. 
Due to \prettyref{prop:linearisation_scalar_valued} (a) we only need to show that $T_{Y}(\widetilde{\delta})$ is surjective and 
its inverse continuous. We do this by proving that 
$T_{Z}(\widetilde{\delta})^{-1}=\widetilde{T}_{Y}(\delta)$ holds.

There is an algebraic isomorphism
$\widetilde{T}\colon\mathcal{F}(\Omega,Y)_{\sigma,b}\to L(Z,Y)$ with 
$\widetilde{T}(f)\circ\widetilde{\delta}=f$ for all $f\in\mathcal{F}(\Omega,Y)_{\sigma}$ where we set 
$\widetilde{T}\coloneqq\widetilde{T}_{Y}$.
Since $\delta\in\mathcal{F}(\Omega,Y)_{\sigma,b}$ by assumption, 
we get the commuting diagram: 
\[
\xymatrix{
\Omega \ar[d]_{\widetilde{\delta}} \ar[r]^{\delta} &  Y  \\
Z \ar[ur]_{\widetilde{T}(\delta)} &
}
\]
Furthermore, there is a topological isomorphism 
$T\coloneqq T_{Z}\colon\mathcal{F}(\Omega,Z)_{\sigma,b}\to L_{b}(Y,Z)$ with $T(f)\circ \delta=f$ for all $f\in\mathcal{F}(\Omega,Z)_{\sigma,b}$ 
by \prettyref{rem:linearisation_predual}. Since $\widetilde{\delta}\in \mathcal{F}(\Omega,Z)_{\sigma,b}$ by assumption, we obtain 
the commuting diagram: 
\[
\xymatrix{
\Omega \ar[d]_{\delta} \ar[r]^{\widetilde{\delta}} &  Z  \\
Y \ar[ur]_{T(\widetilde{\delta})} &
}
\]
Combining both diagrams, we get the commuting diagram:
\[
\xymatrix{
\Omega \ar[d]_{\delta} \ar[r]^{\widetilde{\delta}} &  Z \ar[dl]_{\mathclap{\widetilde{T}(\delta)}\hspace{0.1cm}} \\
Y \ar@<-1ex>[ur]_{T(\widetilde{\delta})} &
}
\]
Now, we note that $T(\widetilde{\delta})\circ\widetilde{T}(\delta)\in L(Z)$, $\widetilde{T}(\delta)\circ T(\widetilde{\delta})\in L(Y)$ and $(T(\widetilde{\delta})\circ\widetilde{T}(\delta))(\widetilde{\delta}(x))=T(\widetilde{\delta})(\delta(x))=\widetilde{\delta}(x)$ as well as $(\widetilde{T}(\delta)\circ T(\widetilde{\delta}))(\delta(x))=\widetilde{T}(\delta)(\widetilde{\delta}(x))=\delta(x)$ for all $x\in\Omega$. Since the span of $\{\widetilde{\delta}(x)\;|\;x\in\Omega\}$ is dense in $Z$ 
and the span of $\{\delta(x)\;|\;x\in\Omega\}$ dense in $Y$ by \cite[Proposition 2.7, p.~1596]{kruse2023a}, 
we obtain that $T(\widetilde{\delta})$ and $\widetilde{T}(\delta)$ are topological isomorphisms 
with $T(\widetilde{\delta})=\widetilde{T}(\delta)^{-1}$. 

Second, we know by \prettyref{prop:linearisation_scalar_valued} (a) that $\widetilde{T}_{\K}^{-1}$ is continuous 
and $T_{Z}(\widetilde{\delta})^{t}=T_{\K}\circ \widetilde{T}_{\K}^{-1}$. Since 
\[
\widetilde{T}_{\K}=(T_{Z}(\widetilde{\delta})^{t})^{-1}\circ T_{\K}=(T_{Z}(\widetilde{\delta})^{-1})^{t}\circ T_{\K}
\]
and $T_{\K}$ as well as $T_{Z}(\widetilde{\delta})$ are topological isomorphisms, we get that $\widetilde{T}_{\K}$ is a 
topological isomorphism, too. This implies that $(Z,\widetilde{T}_{\K})\sim (Y,T_{\K})$. 
\end{proof}

The proof of $T_{Z}(\widetilde{\delta})^{-1}=\widetilde{T}_{Y}(\delta)$ in \prettyref{prop:linearisation_predual_unique} 
is an adaptation of parts of the proof of \cite[Corollary 1, p.~691]{carando2004}.
Further, \prettyref{prop:linearisation_predual_unique} complements \cite[Corollary 1, p.~691]{carando2004} 
where a continuous linearisation $(e,\mathcal{F}_{\ast}(\Omega),L)$ of $\F$ instead of a strong linearisation is considered. 
In comparison to \cite[Corollary 1, p.~691]{carando2004} we do not need that 
$\widetilde{T}_{E}$ is an algebraic isomorphism for all complete locally convex Hausdorff spaces $E$ over $\K$ 
(even though an inspection of its proof tells us that this is not needed there either) and we get more importantly 
that $(Z,\widetilde{T}_{\K})\sim (Y,T_{\K})$. 
The latter equivalence also allows us to show how the continuous linearisation $(e,\mathcal{F}_{\ast}(\Omega),L)$ is related to our 
strong linearisations from \prettyref{sect:linearisation} in many cases.

\begin{cor}\label{cor:carando_zalduendo_predual}
Let $(\F,\tau)$ be a bornological BC-space of $\K$-valued continuous functions on 
a non-empty Hausdorff $gk_{\R}$-space $\Omega$ satisfying $\operatorname{(BBC)}$ and 
$\operatorname{(CNC)}$ for some $\tau_{\operatorname{p}}\leq\widetilde{\tau}$. 
Then the maps $\mathcal{I}_{Z}(e)\colon \F_{\mathcal{B},\widetilde{\tau}}'\to \mathcal{F}_{\ast}(\Omega)$ 
and $L$ are topological isomorphism such that $\mathcal{I}_{Z}(e)^{t}=\mathcal{I}\circ L^{-1}$, 
the triple $(e,\mathcal{F}_{\ast}(\Omega),L)$ is a continuous strong complete barrelled linearisation of $\F$ 
and $(\mathcal{F}_{\ast}(\Omega),L)\sim(\F_{\mathcal{B},\widetilde{\tau}}',\mathcal{I})$.
\end{cor}
\begin{proof}
The triple $(\delta,Y,T_{\K})\coloneqq(\Delta,\F_{\mathcal{B},\widetilde{\tau}}',\mathcal{I})$ is a continuous strong complete 
barrelled linearisation of $\F$ by \prettyref{cor:scb_linearisation}, 
and $\Delta(x)=\delta_{x}\in(\F,\tau)'$ for all $x\in\Omega$ by \prettyref{prop:predual_complete} (a). 
Furthermore, since $(\F,\tau)$ is a BC-space and $\Delta$ continuous, we have  
\[
\FE_{\sigma,b}=\FE_{\sigma}=(\FE_{\sigma}\cap\mathcal{C}(\Omega,E))\eqqcolon\omega\FE 
\]
for any complete locally convex Hausdorff space $E$ by \prettyref{rem:BC_space} 
and \prettyref{thm:linearisation_full}.
Due to \cite[Theorems 2, 3, p.~689--690]{carando2004} $Z\coloneqq\mathcal{F}_{\ast}(\Omega)$ is complete, 
$e\in\omega\mathcal{F}(\Omega,\mathcal{F}_{\ast}(\Omega))$ and for any locally convex Hausdorff space 
$E$ there is an algebraic isomorphism $L_{E}\colon\omega\FE \to L(\mathcal{F}_{\ast}(\Omega),E)$ 
with $L_{E}(f)\circ e=f$ for all $f\in \omega\FE$. Hence it follows from \prettyref{prop:linearisation_predual_unique} with $\widetilde{\delta}\coloneqq e$ and 
$\widetilde{T}_{E}\coloneqq L_{E}$ that the maps 
$\mathcal{I}_{Z}(e)$ and $L$ are topological isomorphism with $\mathcal{I}_{Z}(e)^{t}=\mathcal{I}\circ L^{-1}$, 
the triple $(e,\mathcal{F}_{\ast}(\Omega),L)$ is a continuous strong complete barrelled linearisation of $\F$ and 
$(\mathcal{F}_{\ast}(\Omega),L)\sim(\F_{\mathcal{B},\widetilde{\tau}}',\mathcal{I})$.
\end{proof}

In the case that $(\F,\tau)$ is a Banach space of continuous $\K$-valued functions 
on a topological Hausdorff space $\Omega$ satisfying $\operatorname{(BBC)}$ for some 
$\tau_{\operatorname{p}}\leq\widetilde{\tau}$ it is already known that 
$(e,\mathcal{F}_{\ast}(\Omega),L)$ is a continuous strong Banach linearisation of $\F$ by 
\cite[Theorem 2.2, p.~188]{jaramillo2009}.

In our last results of this section we show how to get rid of the condition that the family 
$(\delta(x))_{x\in\Omega}$ should be linearly independent for the implication (b)$\Rightarrow$(a) 
of \prettyref{prop:strongly_unique_C_predual} and \prettyref{cor:unique_C_predual}, at least if it is a 
family of point evaluation functionals.

\begin{cor}\label{cor:strongly_unique_C_predual_without_lin_ind}
Let $(\F,\tau)$ be a bornological BC-space of $\K$-valued functions on a non-empty set $\Omega$ satisfying 
$\operatorname{(BBC)}$ and $\operatorname{(CNC)}$ for some $\tau_{\operatorname{p}}\leq\widetilde{\tau}$. 
Let $\mathcal{C}$ be a subclass of the class of complete barrelled locally convex Hausdorff spaces 
such that $\mathcal{C}$ is closed under topological isomorphisms and 
$\F_{\mathcal{B},\widetilde{\tau}}'\in\mathcal{C}$ where $\mathcal{B}$ is the family of 
$\tau$-bounded sets.  
Then the following assertions are equivalent. 
\begin{enumerate}
\item[(a)] $(\F,\tau)$ has a strongly unique $\mathcal{C}$ predual. 
\item[(b)] For every predual $(Z,\varphi)$ of $(\F,\tau)$ such that $Z\in\mathcal{C}$ and every $x\in\Omega$ there is a (unique) $z_{x}\in Z$ with $\delta_{x}=\varphi(\cdot)(z_{x})$.
\end{enumerate}
\end{cor}
\begin{proof}
(a)$\Rightarrow$(b) This part follows from \prettyref{prop:strongly_unique_C_predual} since 
$(\Delta,\F_{\mathcal{B},\widetilde{\tau}}',\mathcal{I})$ is a strong complete barrelled linearisation 
by \prettyref{cor:scb_linearisation} (a), $\F_{\mathcal{B},\widetilde{\tau}}'\in\mathcal{C}$ by assumption and 
$\mathcal{I}(f)(\Delta(x))=\mathcal{I}(f)(\delta_{x})=\delta_{x}(f)$ for all $f\in\F$ and $x\in\Omega$. 

(b)$\Rightarrow$(a) Let $(Z,\varphi)$ be a predual of $(\F,\tau)$ such that $Z\in\mathcal{C}$.     
Due to \cite[Proposition 3.21 (a), p.~1605]{kruse2023a} $(\F,\tau)$ satisfies $\operatorname{(BBC)}$ and $\operatorname{(CNC)}$ for 
$\sigma_{\varphi}(\F,Z)$ because $Z$ is complete and barrelled. By assumption for every $x\in\Omega$ there is $z_{x}\in Z$ with $\delta_{x}=\varphi(\cdot)(z_{x})$ and 
thus $\delta_{x}\in \F_{\mathcal{B},\sigma_{\varphi}(\F,Z)}'$ by \cite[Proposition 3.21 (b), (c), p.~1605]{kruse2023a}. 
Hence $(\Delta,\F_{\mathcal{B},\sigma_{\varphi}(\F,Z)}',\mathcal{I}_{\varphi})$ is a strong complete barrelled linearisation of $\F$ 
by \prettyref{cor:scb_linearisation} (a). Further, $\F_{\mathcal{B},\sigma_{\varphi}(\F,Z)}'\in\mathcal{C}$ since $Z$ and 
$\F_{\mathcal{B},\sigma_{\varphi}(\F,Z)}'$ are topologically isomorphic by \cite[Proposition 3.21 (b), p.~1605]{kruse2023a} and 
$\mathcal{C}$ closed under topological isomorphisms. By \prettyref{prop:predual_complete} (a) 
we know that $\Delta(x)=\delta_{x}\in (\F,\tau)'$ for all $x\in\Omega$.
Applying \prettyref{thm:linearisation_full} combined with \prettyref{rem:BC_space} to the triple 
$(\Delta,\F_{\mathcal{B},\sigma_{\varphi}(\F,Z)}',\mathcal{I}_{\varphi})$, where 
\[
\mathcal{I}_{\varphi}\colon (\F,\tau)\to (\F_{\mathcal{B},\sigma_{\varphi}(\F,Z)}',\beta)_{b}',\;
f\longmapsto [f'\mapsto f'(f)],
\]
we get that for all complete locally convex Hausdorff spaces $E$ 
there is a topological isomorphism 
$\widetilde{T}_{E}\colon \FE_{\sigma} \to L_{b}(\F_{\mathcal{B},\sigma_{\varphi}(\F,Z)}',E)$ with 
$\widetilde{T}_{E}(f)\circ\Delta=f$ for all $f\in\FE_{\sigma}$, where $\widetilde{T}_{\K}=\mathcal{I}_{\varphi}$ 
by \cite[Proposition 3.21 (b), p.~1605]{kruse2023a}. 
Thus it follows from \prettyref{prop:linearisation_predual_unique} with $\widetilde{\delta}\coloneqq\Delta$ and 
$Y\coloneqq \F_{\mathcal{B},\widetilde{\tau}}'$ that 
$(\F_{\mathcal{B},\sigma_{\varphi}(\F,Z)}',\mathcal{I}_{\varphi})\sim(\F_{\mathcal{B},\tau}',\mathcal{I})$. 
We also know that $(\F_{\mathcal{B},\sigma_{\varphi}(\F,Z)}',\mathcal{I}_{\varphi})\sim (Z,\varphi)$ 
by \cite[Proposition 3.21 (b), p.~1605]{kruse2023a}, which implies 
$(Z,\varphi)\sim(\F_{\mathcal{B},\tau}',\mathcal{I})$. Therefore $(\F,\tau)$ has a strongly unique $\mathcal{C}$ predual 
by \prettyref{prop:strongly_unique_equivalent}.
\end{proof}

\begin{cor}\label{cor:unique_C_predual_without_lin_ind}
Let $(\F,\tau)$ be a bornological BC-space of $\K$-valued functions on a non-empty set $\Omega$ satisfying 
$\operatorname{(BBC)}$ and $\operatorname{(CNC)}$ for some $\tau_{\operatorname{p}}\leq\widetilde{\tau}$. 
Let $\mathcal{C}$ be a subclass of the class of complete barrelled locally convex Hausdorff spaces 
such that $\mathcal{C}$ is closed under topological isomorphisms and 
$\F_{\mathcal{B},\widetilde{\tau}}'\in\mathcal{C}$ where $\mathcal{B}$ is the family of 
$\tau$-bounded sets.  
Then the following assertions are equivalent. 
\begin{enumerate}
\item[(a)] $(\F,\tau)$ has a unique $\mathcal{C}$ predual. 
\item[(b)] For every predual $(Z,\varphi)$ of $(\F,\tau)$ such that $Z\in\mathcal{C}$ there is 
a topological isomorphism $\psi\colon(\F,\tau)\to Z_{b}'$ such that for every $x\in\Omega$ there is a (unique) $z_{x}\in Z$ with $\delta_{x}=\psi(\cdot)(z_{x})$.
\end{enumerate}
\end{cor}
\begin{proof}
This statement follows from \prettyref{cor:unique_C_predual} and the 
proof of \prettyref{cor:strongly_unique_C_predual_without_lin_ind} with $\varphi$ replaced by $\psi$ and 
\prettyref{prop:unique_equivalent} instead of \prettyref{prop:strongly_unique_equivalent} in the end. 
\end{proof}

%
%

\bibliography{biblio_linearisation_uniqueness}
\bibliographystyle{plainnat}
\end{document}